\documentclass[reqno,11pt]{amsart}
\usepackage[top=2.5cm,bottom=2.5cm,left=2.5cm,right=2.5cm]{geometry}
\usepackage{amsthm,amsmath,amssymb,dsfont}
\usepackage{mathrsfs,amsfonts,functan,extarrows,mathtools}
\usepackage{indentfirst}
\usepackage{appendix}
\usepackage{xcolor}
\usepackage{latexsym, enumerate,graphicx}
\usepackage{float}
\usepackage[colorlinks, linkcolor=red, citecolor=blue]{hyperref}
\usepackage{cite}

\allowdisplaybreaks[4]

\numberwithin{equation}{section}
\newtheorem{thm}{Theorem}[section]
\newtheorem{lem}[thm]{Lemma}

\newtheorem{rem}{Remark}[section]

\begin{document}

	
	\def\esssup{\mathop{\rm ess\, sup}}
	
	\title[Global well-posedness 
    of 3D Euler equations with damping]
	{Well-Posedness and Asymptotic Decay of Solutions\\[1.5mm] to the Three-Dimensional Euler Equations with Damping}
	
	\author[G.-Q. Chen]{Gui-Qiang G. Chen}
	\address[G.-Q. Chen]{Mathematical Institute, University of Oxford, Oxford OX2 6GG, UK}
	\email{gui-qiang.chen@maths.ox.ac.uk}

	\author[F. Huang]{Feimin Huang}
	\address[F. Huang]{Academy of Mathematics and Systems Science, Chinese Academy of Sciences, Beijing 100190, P. R. China, and School  of  Mathematical  Sciences,  University  of  Chinese  Academy  of  Sciences,  Beijing 100049, P. R. China}
	\email{fhuang@amt.ac.cn}
	
	\author[H. Tang]{Houzhi Tang}
	\address[H. Tang]{School of Mathematics and Statistics, Anhui Normal University, Wuhu 241002, P. R.  China}
	\email{houzhitang@ahnu.edu.cn}
	
	\author[S. Zhang]{Shuxing Zhang}
	\address[S. Zhang]{School of Mathematical Sciences, Jiangsu University, Zhenjiang 212013, P. R. China}
	\email{zhangsx@ujs.edu.cn}
	
	\author[W. Zou]{Weiyuan Zou}
	\address[W. Zou]{College of Mathematics and Physics,
		Beijing University of Chemical Technology, Beijing 100029, P. R. China, and Mathematical Institute, University of Oxford, Oxford OX2 6GG, UK}
	\email{zwy@amss.ac.cn}

	\date{}

\begin{abstract}
The global well-posedness 
of the multi-dimensional compressible Euler equations with damping remains 
a longstanding open problem. 
This problem has been partially resolved in the isentropic regime ({\it i.e.}, 
the adiabatic exponent \(\gamma>1\)) for small smooth initial data (see \cite{WY, STW}). 
In this paper, we establish the global well-posedness and 
asymptotic decay of smooth solutions of the Cauchy problem of 
the three-dimensional compressible Euler equations with damping 
for the isentropic regime \(\gamma>1\) and the isothermal regime \(\gamma=1\), 
allowing for partially large initial data. More precisely, 
the \(L^2\)-norm of the initial data is allowed to be large, while the third-order Sobolev norm of the initial data is assumed to be small.
For the isentropic case, we develop a new analytical framework 
in which all required
{\it a priori} estimates of solution $(\rho,u)$ can be derived under the condition that
$\int_0^T \big( \|\nabla\rho\|_{L^\infty} + \|\nabla u\|_{L^\infty} \big) \, \mathrm{d}t$
remains sufficiently small.
Moreover, we obtain the optimal algebraic decay rates of global solutions. 
Furthermore, we study the isothermal limit of solutions of 
the isentropic regime as $\gamma \to 1$, and 
establish the global well-posedness and asymptotic decay of solutions to the isothermal Euler equations 
with damping.

\end{abstract}
\maketitle
	
\section{Introduction}
	
	\hspace{2em}The three-dimensional Euler equations with damping can be written as
	\begin{equation}\label{euler}
		\left\{
		\begin{aligned}
			& \partial _t \rho  + \mathrm{div}(\rho u) = 0, \\
			& \partial _t (\rho u) + \mathrm{div}(\rho u\otimes u) + \nabla P(\rho)  =-\rho u,\\
		\end{aligned}
		\right.
	\end{equation}
	which model compressible gas flow through a porous medium. Here, the unknown functions $\rho=\rho(t,x)$ and $u=u(t,x)$ denote the density and velocity of the fluid in the domain $(t,x)\in \mathbb{R}_{+}\times\mathbb{R}^3$, respectively.  The pressure $P:[0,\infty)\times\mathbb{R}^3\to\mathbb{R}$ is given by the
barotropic pressure law
$$P=P(\rho)=a_0\rho^\gamma
\qquad \mbox{\text{for $a_0>0$}},$$
where $\gamma>1$ is the adiabatic exponent, while the case $\gamma=1$
corresponds to the isothermal case.  
Without loss of generality, we normalize $a_0=1$ for simplicity.
System \eqref{euler} is equipped with the initial data:
	\begin{equation}\label{initial}
		(\rho,u)(0,x)=(\rho_{0},u_{0})(x),
	\end{equation}
	and the far-field states:
	\begin{equation}\label{far}
		\lim_{|x|\rightarrow \infty}(\rho,u)(t,x)=(\rho_*,0),
	\end{equation}
	where $\rho_*>0$ is a given constant.
	
	The global existence of solutions to the $\gamma$-law Euler equations is a challenging problem
    because the solution develops singularities, such as shocks, in finite time, regardless of how small and smooth the initial data are. For the one-dimensional case, Nishida \cite{N68} used the Glimm difference scheme to prove the global existence of BV solutions with large initial data for $\gamma = 1$ as the gas is away from the vacuum. This result was extended to the case $\gamma>1$ by Nishida-Smoller \cite{NS} when $\gamma$ is close 1.
	When vacuum occurs, DiPerna \cite{Di} first established the global existence of solutions with large initial data in $L^{\infty}$ for $\gamma=1+\frac{2}{2n+1}$, where $n\geq 2$ denotes the number of degrees of freedom
	of the molecules.
	Later, the existence problem was solved by  Ding-Chen-Luo\cite{DCL-1,DCL-2} for $\gamma\in(1,\frac{5}{3}]$,  Lions-Perthame-Tadmor \cite{LPT} 
    for $\gamma\in [3, \infty)$,
    and Lions-Perthame-Souganidis \cite{LPS} for $\gamma\in (\frac{5}{3}, 3)$.  
    The $\gamma=1$ case was treated by Huang-Wang \cite{HW}.
	On the other hand, the global existence of solutions with
    large initial data in the multidimensional (M-D) case remains open.
	
	It is known that the damping can improve the regularity of solutions; see \cite{HL,HLuo,HS,HuangP,HuangP1,N78,STW,HWU,JQ2012,XZ2024}. That is, the damping  prevents the development of singularities in some sense and leads to the global existence of smooth solutions for sufficiently small initial data.\;However, the damping cannot prevent the formation of singularities for large initial data; see \cite{LY,LY-2}. The entropy solution has to be considered; see \cite{Daf,DP,DCL-3}.
	
	It is fundamental to further study the regularity of the damping structure. In this paper, we establish the global well-posedness and asymptotic decay of smooth solutions of the Cauchy problem 
    of the M-D compressible Euler equations with damping for some partially large initial data, {\it i.e.},  $\|(\rho_0-\rho_*,u_0)\|_{L^2}$ is allowed to be large, but $\|D^3(\rho_0-\rho_*,u_0)\|_{L^2}$ is necessarily small. Our results consist of two parts. In the first part, concerning the isentropic regime ($\gamma >1$), the precise statement is following:
	\begin{thm}\label{main}
		Let the initial data $(\rho_0-\rho_*, u_0)\in H^3(\mathbb{R}^3) $ satisfy $0<\underline{\rho}\leq \rho_0\leq \bar{\rho}<\infty$ for some positive constants $\underline{\rho}$ and $\bar{\rho}$,
		\begin{equation}\label{H3}
			\|(\rho_0-\rho_*,u_0)\|_{L^2}\leq M_0 \;,\qquad \qquad \|D^3 (\rho_0-\rho_*,  u_0)\|_{L^2}\leq \delta_0,
		\end{equation}
		where  $\delta_0>0$ is small and $0<M_0\leq \delta_0^{-\frac{1}{11}}$. Then the Cauchy problem  \eqref{euler}--\eqref{far} for the isentropic case, i.e., $\gamma > 1$,  admits a unique global classical solution
		$(\rho-\rho_*, u)\in C([0,\infty);H^3(\mathbb{R}^3))$ satisfying
		\begin{align}
			&\|D^{k}(\rho-\rho_*)(t)\|_{L^2}\leq C(1+t)^{-\frac{k}{2}},\label{rhodecay} \\
			&\|D^{k} u(t)\|_{L^2}\leq C(1+t)^{-\frac{k+1}{2}},\quad\label{udecay} 
		\end{align}
		for $ k=0,1,2,3$, where   $C=C(M_0,\underline{\rho},\bar{\rho},\delta_0)$ is a positive constant independent of $t\ge 0$.
	\end{thm}

The equation of state for isothermal gases can be regarded as the limit of that for isentropic gases as $\gamma \to 1$. Hence, the second main result of this paper addresses the isothermal limit of the isentropic damped Euler system \eqref{euler}-–\eqref{far}.

\begin{thm}[Isothermal Limit]\label{thm:gamma_to_1}
Assume that the conditions in \rm{Theorem} \ref{main} hold.  
Then, as $\gamma \to 1$, the solutions of the Cauchy 
problem \eqref{euler}--\eqref{far} for the isentropic regime ($\gamma > 1$) strongly converge to the solution of the corresponding problem for $\gamma=1$. 
Moreover, the solution of  the Cauchy problem \eqref{euler}--\eqref{far} 
for $\gamma=1$ is a unique global classical solution
$(\rho-\rho_*, u)\in C([0,\infty);H^3(\mathbb{R}^3))$ satisfying \eqref{rhodecay}--\eqref{udecay}.
\end{thm}

\begin{rem}
\rm{Theorems} \ref{main}--\ref{thm:gamma_to_1} admit the following initial data:
		\begin{align}
			\rho_0=\rho_*+\varepsilon^{\frac{5}{11}}e^{-\varepsilon^{\frac{8}{11}} |x|^2},\qquad \qquad 	u_0=\varepsilon^{\frac{5}{11}}e^{-\varepsilon^{\frac{8}{11}} |x|^2},
		\end{align}
		where $\varepsilon<\delta_0$ is  small. Indeed, $\|(\rho_0-\rho_*,u_0)\|_{L^2}= O(\varepsilon^{-\frac{1}{11}}) \rightarrow \infty \; \text{as $\varepsilon \rightarrow 0$}$, thus the $L^2$-norm of the initial data can be large.
\end{rem}	
\begin{rem}
{\rm{The decay rates  obtained in Theorems \ref{main}--\ref{thm:gamma_to_1} are optimal, which are the same as the ones of the solution to the corresponding linear system.}}
\end{rem}

The strategies of the proofs are outlined as follows: Firstly, for the isentropic case $\gamma>1$, we introduce a new function
\begin{equation}\label{transformphi}
\phi=\frac{2}{\gamma-1}\big(\sigma(\rho)-\sigma(\rho_*)\big),
\end{equation}
where $\sigma(\rho)=\sqrt{P'(\rho)}$ is the sound speed and $\sigma(\rho_*)=\bar{\sigma}$ corresponds to the sound speed at the background density $\rho_*$.  
Then system \eqref{euler}--\eqref{far} can be reformulated as a symmetric hyperbolic system with partial damping:
	\begin{align}\label{eq:isentropic}
	\begin{cases}
			\partial_t \phi+\bar{\sigma}\, {\mathrm{div}} u=-u \cdot \nabla \phi-\frac{\gamma-1}{2} \phi  \,{\mathrm{div}} u, \\[1mm]
			\partial_t u+\bar{\sigma}\, \nabla \phi+ u=-u \cdot \nabla u-\frac{\gamma-1}{2} \phi \nabla \phi .
		\end{cases}
	\end{align}
For this system, the key observation is that all  $L^2$-norm of the $k$-order derivatives of the solution can be controlled by the $k$-order derivative of the initial data. More precisely, there exists a constant $C>0$ such that
	\begin{equation}\label{d2}
		\|D^k(\phi, u)(t)\|_{L^2}\leq C\|D^k(\phi_0, u_0)\|_{L^2} \qquad \text{for \,$k=0,1,2,3,$} 
	\end{equation}
provided that 
\begin{equation} \label{aprior}
\int_0^T\big(\|\nabla \phi\|_{L^{\infty}}+\|\nabla u\|_{L^{\infty}}\big)\,\mathrm{d}t 
\end{equation}
is small. Then, by the analysis of the Green function $G_{ij}$ for the linearized system of \eqref{eq:isentropic} (see \eqref{xianxing} below), we obtain that, for any test function $\zeta$,
\begin{equation}\label{GL1}
\|D^k G_{ij}^{\ell}\ast \zeta\|_{L^2} \le 
\begin{cases}
C(1+t)^{-\frac{k}{2}}\|\zeta\|_{L^2} & \text{for \;$i+j=2$},\\[1mm]
C(1+t)^{-\frac{k+1}{2}}\|\zeta\|_{L^2}& \text{for \;$i+j=3$},\\
C(1+t)^{-\frac{k+2}{2}}\|\zeta\|_{L^2}
			\quad & \text{for \;$i+j=4$},
\end{cases}
\end{equation}
	where $G^{\ell}_{ij}$ corresponds to the low-frequency part. Based on \eqref{GL1} and the Duhamel principle, we can further obtain $$
	\|D^k\phi^\ell\|_{L^2}\le C(1+t)^{-\frac{k}{2}} \qquad \text{for \;$ k=1,2,3.$}$$
	 Using the energy method again, we have	
     \begin{equation}
		\frac{\mathrm{d}}{\mathrm{d}t}
		\|(D^k \phi,D^k u)\|_{H^{3-k}}^2
		+ \|(D^k \phi,D^k u)\|_{H^{3-k}}^2
		\leq C \|D^k \phi^{\ell}\|_{L^2}^2 \qquad \text{for \;$ k=1,2,3,$}
	\end{equation}
	which implies
	\begin{equation}
		\|D^{k}(\phi,u)(t)\|_{L^2}\leq C(1+t)^{-\frac{k}{2}} \qquad \text{for \;$ k=1,2,3.$}\label{d0}
	\end{equation}
On the other hand, since $u$ is only related to  $G_{21}$ and $G_{22}$, \eqref{GL1} hints that $\|D^ku\|_{L^2}$ should have the higher decay rate than $(1+t)^{-\frac{k}{2}}$. Indeed,  we decompose $u$ into the low-frequency part $u^{\ell}$ and the high-frequency part $u^{h}$. With the help of the damping structure, we obtain from   \eqref{GL1} and the Duhamel principle that
$$\|D^ku^\ell\|_{L^2}\le C(1+t)^{-\frac{k+1}{2}}.$$
Meanwhile, the energy estimate of the high-frequency part of \eqref{eq:isentropic} implies that $\|D^ku^h\|_{L^2}$ has better decay rate than the one for the low frequency part. Hence, we obtain a better decay rate for $D^ku$, which is optimal, since it is the same as that of the linearized equation, {\it i.e.},
	\begin{equation}
		\|D^k u(t)\|_{L^2}\leq C(1+t)^{-\frac{k+1}{2}} \qquad 
        \text{for \;$k=1,2,3.$}\label{d1}
	\end{equation}

	Finally, it remains to show that $\int_0^T(\|\nabla \phi\|_{L^{\infty}}+\|\nabla u\|_{L^{\infty}})\,\mathrm{d}t $ is small. This can be directly verified  by the optimal decay rates \eqref{d0}--\eqref{d1} under the condition that $\|D^3(\rho_0-\rho_*,u_0)\|_{L^2}$ is small, but $\|(\rho_0-\rho_*,u_0)\|_{L^2}$ can be large.

Moreover, the method for the isentropic case cannot be directly applied to the isothermal case since transform \eqref{transformphi} is singular when $\gamma=1$. This is why we need to study the isothermal limit of the isentropic damped Euler equations. The method of studying the isothermal limit is based on the compactness argument. Firstly, in the proof of Theorem \ref{main}, we can obtain a crucial by-product: the uniform estimates of solutions with respect to $\gamma\in (1,3]$. Secondly, by applying the Aubin–Lions lemma, we obtain the convergence of the solution sequence as $\gamma \to1$, which allows us to rigorously justify the isothermal limit. The global well-posedness and asymptotic decay of solutions 
to the isothermal damped Euler equations can be regarded as a direct consequence of this isothermal limit.

Furthermore, we point out that the global well-posedness and asymptotic decay of
solutions to the isothermal damped Euler equations can be established by a similar argument used in Theorem \ref{main}. To be specific, 
according to transform \eqref{transformphi}, we take the limit to obtain
\begin{align*}
\underset{\gamma \to 1}{\lim} \phi(\rho, \gamma)= \underset{\gamma \to 1}{\lim} \frac{2\sqrt{\gamma}\big(\rho^{\frac{\gamma-1}{2}} - \rho_*^{\frac{\gamma-1}{2}}\big)}{\gamma - 1}= \ln \rho - \ln \rho_*=:\hat{\phi},
\end{align*} 
and then pass to the limit in \eqref{eq:isentropic} to obtain
	\begin{equation}\label{au0}
		\left\{
		\begin{aligned}
			&\partial_t \hat{\phi} +u\cdot \nabla \hat{\phi}+\mathrm{div} u=0,\\
			&\partial_t u+u\cdot \nabla u+\nabla \hat{\phi} +u=0,
		\end{aligned}
		\right.
	\end{equation}
for $(\hat{\phi}, u)$.
It is direct to verify that this system has the same form as the following compressible isothermal Euler equations with damping:	
	\begin{equation}\label{dengwen-E}
		\left\{
		\begin{aligned}
			& \partial _t \rho  + \mathrm{div}(\rho u) = 0, \\
			& \partial _t (\rho u) + \mathrm{div}(\rho u\otimes u) + \nabla \rho  =-\rho u.\\
		\end{aligned}
		\right.
	\end{equation}
Then, after necessary modifications, we can also use a similar method used in proving Theorem \ref{main} to obtain the well-posedness and asymptotic decay
of solutions to system \eqref{dengwen-E}.
In fact, system \eqref{au0} has a simpler form than \eqref{eq:isentropic} due to fewer nonlinear terms after the new transform: $\hat{\phi}=\mathrm{ln} \rho-\mathrm{ln}\rho_*$. Then we can apply the method developed in \rm{Theorem} \ref{main}.
		
	The rest of the paper is organized as follows:
	In Section  \ref{sec.2}, we present some notation and important lemmas, which are used throughout the paper.
	In Section \ref{s1}, we establish the global well-posedness and asymptotic
    decay of the solution of the Cauchy problem \eqref{euler}--\eqref{far} for $\gamma>1$.
	In Section \ref{s2}, we obtain the isothermal limit of the isentropic case.
	
	\section{Preliminary}\label{sec.2}
	\hspace{1em}We now introduce the notation used throughout the present paper.
	The Greek letter $\alpha$ is used to denote a multi-index $\alpha=(\alpha_1,\alpha_2,\alpha_3)$ for integers $\alpha_i\geq0$, $i=1,2,3$.
	We denote by $D^{\alpha}=\partial^{\alpha_1}_{x_1}\partial^{\alpha_2}_{x_2}\partial^{\alpha_3}_{x_3}$
	the partial derivative of order $|\alpha|=\alpha_1+\alpha_2+\alpha_3.$
	In particular, we use $D^k$ to denote $D^{\alpha}$ with $|\alpha|=k$.
	$L^p(\mathbb{R}^3)$ and $W^{k,p}(\mathbb{R}^3)$ denote the usual Lebesgue and Sobolev spaces on $\mathbb{R}^3$ with norms $\|\cdot\|_{L^p}$ and $\|\cdot\|_{W^{k,p}}$, respectively.  When $p=2$, we denote $W^{k,p}(\mathbb{R}^3)$ by $H^k(\mathbb{R}^3)$ with norm $\|\cdot\|_{H^k}$. We denote by $C$ a generic positive constant that may vary in different estimates, $\tilde{C}=e^{C_1+C_2(\gamma-1)}$ to present the dependence of $\gamma-1$. $f\sim g$ means that there exists two positive constants $C_1\leq C_2$ such that $C_1|g|\leq|f|\leq C_2|g|$.
	The Fourier transform of a function $f$ is denoted by $\hat{f}$ or $\mathscr{F}[f]$ satisfying
	$$
    \hat{f}(\xi)=\mathscr{F}[f](\xi)
	=\int_{\mathbb{R}^3} f(x)e^{-i2\pi x\cdot\xi}\mathrm{d}x \qquad \text{for \;$\xi\in\mathbb{R}^3$}.
    $$
	Correspondingly, the inverse Fourier transform of $f$ is denoted by $\check{f}$ or $\mathscr{F}^{-1}[f]$ such that $$\check{f}(x)=\mathscr{F}^{- 1}[f](x)=\int_{\mathbb{R}^3} f(\xi)e^{i2\pi\xi\cdot x}\mathrm{d}\xi 
    \qquad \text{for \;$x\in\mathbb{R}^3$}.
    $$
	The low-high frequency projection operators $P_{1}$ and $P_{\infty}$  on $L^{2}$ are defined by
	\begin{equation}\label{p1pinf}
		P_{j}f=\mathscr{F}^{- 1}(\hat\chi_{j}\mathscr{F}[f])\qquad \mbox{for \;$j=1$~\text{or}~$\infty$},
	\end{equation}
	where $\hat{\chi}_{1},\;\hat{\chi}_{\infty}\in C^{\infty}(\mathbb{R}^{3})$  are the cut-off functions defined by
	\begin{equation*}
		\hat\chi_{1}(\xi)=\left\{
		\begin{array}{l}
			1 \quad\,\, \text{for $|\xi|\leq r_{0}$},\\
			0 \quad\,\, \text{for $|\xi|\geq R_0$},
		\end{array}
		\right.
		\qquad  \quad \hat\chi_{\infty}(\xi)=1-\hat\chi_{1}(\xi),
	\end{equation*}
	for constants $r_0$ and $R_0$ satisfying $0<r_0<R_0$.
	According to this decomposition,  a function  $f$ can be divided into two parts:
	\begin{align}\label{010301}
		f=f^{\ell}+f^{h},
	\end{align}
	where $f^{\ell}=P_1 f$ and $f^{h}=P_{\infty}f$ represent the low-frequency and high-frequency parts of $f$, respectively.
	
	Next, we collect some elementary inequalities and important lemmas that are used
	extensively in this paper.
	
	\begin{lem}[Moser--type calculus inequalities, {\it cf.} \cite{MB}]\label{jiaohuan}
		Let $s\in \mathbb{Z}^+$. Assume that $u,v\in H^s(\mathbb{R}^3)\cap L^{\infty}(\mathbb{R}^3)$. Then, for any multi-index $\alpha$ with $|\alpha|\leq s$,
		\begin{equation}
			\|D^\alpha(uv)\|_{L^2}
			\leq C\big(\|u\|_{L^{\infty}}\|D^s v\|_{L^2}
			+\|v\|_{L^{\infty}}\|D^s u\|_{L^2}\big).
		\end{equation}
		Moreover, if $D u\in L^{\infty}(\mathbb{R}^3)$, then
		\begin{equation}
			\|[D^\alpha, u]v\|_{L^2}
			\leq C\big(\|D u\|_{L^{\infty}}\|D^{s-1} v\|_{L^2}
			+\|v\|_{L^{\infty}}\|D^s u\|_{L^2}\big),
		\end{equation}
		where\, $[D^\alpha, u]v :=D^\alpha(uv)-u D^\alpha v$.
	\end{lem}
	\begin{lem}[Gagliardo--Nirenberg's inequality, {\it cf.}\cite{MB}]\label{GNlem}
		Let $l,s,$ and $k$ be any real numbers satisfying $0\leq l,s<k$, and let  $p, r, q \in [1,\infty]$ and $\frac{l}{k}\leq\theta\leq 1$ such that
		\begin{equation*}
			\frac{l}{3}-\frac{1}{p}=\big(\frac{s}{3}-\frac{1}{r}\big)(1-\theta)
			+\big(\frac{k}{3}-\frac{1}{q}\big)\theta.
		\end{equation*}
		Then  
		\begin{equation}\label{GN}
			\|D^l u\|_{L^p}\leq C \|D^s u\|_{L^r}^{1-\theta}\|D^k u\|_{L^q}^{\theta} \qquad \text{for any $u\in W^{k,q}(\mathbb{R}^3)$}.
		\end{equation}	
	\end{lem}

	The derivatives of low-frequency and high-frequency parts of some function $f$ satisfy the following properties,
	which are direct to verify by the definitions of low-high frequency projection operators and the Parseval equality.
	\begin{lem}\label{kongzhi}
		For any $0\leq n\leq m$,
		\begin{equation}\label{gaodi}
			\|D^{m} U^\ell\|_{L^2}\leq C\|D^{n} U^\ell\|_{L^2},
			\quad \|D^{n}U^h\|_{L^2}\leq C\|D^{m}U^h\|_{L^2},
		\end{equation}
		and
		\begin{equation}\label{gaodixiao}
			\|D^{k} U^\ell\|_{L^2}\leq C\|D^{k} U\|_{L^2},
			\quad\|D^{k}U^h\|_{L^2}\leq C\|D^{k}U\|_{L^2} \qquad \text{for any \,$k \geq0$}.
		\end{equation}
	\end{lem}

Next, we present the Aubin-Lions lemma, {\it cf.} \cite{Simon}, for analyzing the isothermal limit of the isentropic damped Euler equations.

\begin{lem}[Aubin--Lions Lemma]\label{AL}
Assume that $X$, $B$, and $Y$ are Banach spaces and satisfy $X \subset B \subset Y $, with the embedding $X \hookrightarrow B$ being compact.
Define
\begin{align*}
    W = \big\{ f \,:\, f \in L^p (0, T; X), \; \partial_t f \in L^q (0, T; Y) \big\}\qquad 
    \text{for \;$1 \leq p, q \leq \infty$}.
\end{align*}
If $1 \leq p < \infty$ and $q = 1$, then $W$ is relatively compact in $L^p(0,T;B)$.
If $p = \infty$ and $q > 1$, then $W$ is relatively compact  in $C ( [0, T]; B)$.
\end{lem}
 	
	Finally, we show the local existence of the Cauchy problem  \eqref{euler}--\eqref{far}.
	Since \eqref{euler} is a symmetrizable hyperbolic system, by selecting the appropriate symmetrizer and following the proof in Section 2.1 of \cite{M}, we can establish the following local existence of the solutions.
	\begin{thm}\label{existence theorem}
		Assume that the initial data $(\rho_0-\rho_*, u_0)\in H^3(\mathbb{R}^3)$ satisfy
		\begin{equation*}
			\|(\rho_0-\rho_*, u_0)\|_{H^3}\leq M_* \,,\qquad 
		 0<\underline{\rho}\leq \rho_0\leq \bar{\rho}<\infty
		\end{equation*}
		for some positive constants $M_*$, $\underline{\rho}$, and $\bar{\rho}$.
		Then there exists a local time $T_*=T_*(M_*, \underline{\rho}, \bar{\rho})>0$ such that the Cauchy problem \eqref{euler}--\eqref{far} admits a unique classical solution $(\rho - \rho_*,u)\in C([0,T_*];H^3(\mathbb{R}^3))$.
	\end{thm}
	
	\section{Global Well-Posedness And Asymptotic Decay
    For The Isentropic Case: $\gamma>1$}\label{s1}
	\subsection{Regularity Criterion}\label{sec.3}
	In this section, we establish the following regularity criterion for the Cauchy problem \eqref{euler}--\eqref{far}.
	
	\begin{thm}\label{thm3.1}
		Assume that the initial data $(\rho_0-\rho_*, u_0)\in H^3(\mathbb{R}^3)$.
		For any $T>0$, let $(\rho,u)\in C([0,T];H^3(\mathbb{R}^3))$ be a local solution of the Cauchy problem \eqref{euler}--\eqref{far} and $\phi$ as defined in \eqref{transformphi}.
		If
		\begin{equation}\label{priori}
			\int_0^T\big(\|\nabla \phi\|_{L^{\infty}}+\|\nabla u\|_{L^{\infty}}\big)\,\mathrm{d}t\leq \delta
		\end{equation}
		holds for some small constant $\delta>0$, then the solution $(\rho,u)$ can be extended beyond $T$.
	\end{thm}
	\begin{proof}
We prove the theorem under the {\it a priori} assumption \eqref{priori}. 
The proof is divided into three steps: first, we derive uniform upper and lower bounds for the density; second, we establish an $H^3$ energy estimate for $(\phi,u)$; finally, we use the equivalence between $\phi$ and $\rho-\rho_*$ to obtain the corresponding estimate for $(\rho-\rho_*,u)$.
    
\smallskip
{\bf 1.  Uniform bounds for the density.}
Under ssumption \eqref{priori}, the density $\rho(t,x)$ satisfies
		\begin{equation}\label{rho}
			e^{-\delta}\underline{\rho} \leq\rho(t,x)\leq e^{\delta}\overline{\rho} 
            \qquad \text{for any $t\in[0,T]$}.
		\end{equation}
     This can be derived as follows:
		Define the particle trajectory $X(\tau;s,y)$ passing through $(s,y)$ by
		\begin{equation*}
			\left\{
			\begin{aligned}
				&\frac{\mathrm{d}X(\tau;s,y)}{\mathrm{d}\tau}=u(\tau; s,y),\\
				&X(s;s,y)=y.
			\end{aligned}
			\right.
		\end{equation*}
		It follows from \eqref{euler}$_1$ that, along the particle trajectory $X(\tau;s,y)$, the directional derivative of the density is
		\begin{equation*}
		\frac{\mathrm{d}\rho(\tau,X(\tau;s,y))}{\mathrm{d}\tau}=-\rho(\tau,X(\tau;s,y))  \mathrm{div}\,u(\tau,X(\tau;s,y)).
		\end{equation*}
		Solving the above equation yields 
		\begin{equation*}
			\rho(t,X(t;s,y))=\rho(0,X(0;s,y))e^{-\int_0^t  \mathrm{div}\, u(\tau,X(\tau;s,y))\mathrm{d} \tau} \qquad \text{for any $t\in[0,T]$}.
		\end{equation*}
		Then, combining \eqref{priori} and $\rho_0\in[\underline{\rho},\bar{\rho}]$, we obtain \eqref{rho}.   

\smallskip
{\bf 2. Energy estimates for $(\phi,u)$.}
We next establish an $H^3$ estimate for $(\phi,u)$. After a normalization, we may assume without loss of generality that $\rho_*=1$ in Section \ref{s1}. By the uniform bounds in \eqref{rho} and the definition of $\phi$ in \eqref{transformphi},   system \eqref{euler}--\eqref{far} can be reformulated as
	\begin{align}
		\left\{\begin{array}{l}
			\partial_t \phi+\sqrt{\gamma} \,{\mathrm{div}} u=-u \cdot \nabla \phi-\frac{\gamma-1}{2} \phi  \mathrm{div}\,u, \\[1mm]
			\partial_t u+\sqrt{\gamma} \,\nabla \phi+ u=-u \cdot \nabla u-\frac{\gamma-1}{2} \phi \nabla \phi,
		\end{array}\right.\label{au1}
	\end{align}
	with the initial data:
	\begin{equation}\label{initial-au}
		(\phi,u)|_{t=0}=(\phi_0, u_{0})\qquad \,\,
        \mbox{for  $\phi_0= \frac{2}{\gamma-1}(\sigma(\rho_0)-\sigma(1))$},
	\end{equation}
	and the far-field states:
	\begin{equation}\label{far-au}
		\lim_{|x|\rightarrow \infty}(\phi,u)=(0,0).
	\end{equation}
		
Under assumption \eqref{priori}, we have
		\begin{equation}\label{aH3}
			\sup\limits_{t\in[0,T]}\|(\phi,u)(t)\|^2_{H^3}+\int_0^T\|u(t)\|^2_{H^3} \,\mathrm{d}t \leq  e^{(C+(\gamma-1)\hat{C})}\|(\phi_0,u_0)\|^2_{H^3},
		\end{equation}
        where $C$  and $\hat{C}$ are positive constants independent of $\gamma-1$ and time. 
        
	Indeed, we first take the $L^2$-inner product of \eqref{au1} with $(\phi,u)$, which yields
		\begin{align}
			&\frac{1}{2}\frac{\mathrm{d}}{\mathrm{d}t}\|(\phi,u)\|^2_{L^2}+\|u\|^2_{L^2}\notag\\
			&=- \int_{\mathbb{R}^3} \phi\, u\cdot\nabla\phi\, \mathrm{d}x
			- \int_{\mathbb{R}^3} (u\cdot\nabla u) \cdot u\, \mathrm{d}x
			- \frac{\gamma-1}{2} \int_{\mathbb{R}^3} \big(\phi^2 \mathrm{div}u+u\cdot\phi\nabla\phi \, \big)\mathrm{d}x \notag\\
			&\leq \frac12\|\mathrm{div}u\|_{L^\infty}\|\phi\|^2_{L^2}
			+\|\nabla u\|_{L^\infty}\|u\|^2_{L^2}+\frac{\gamma-1}{4}\|\nabla u\|_{L^\infty}\|\phi\|^2_{L^2} \notag\\
			&\leq (\frac32+\frac{\gamma-1}{4})\|\nabla u\|_{L^\infty}(\|\phi\|^2_{L^2}+\|u\|^2_{L^2}). \label{al21}
		\end{align}
		For higher-order estimates, applying $D^{k}$, $k=1,2,3$,  to \eqref{au1} yields
		\begin{equation}\label{akuk1}
			\left\{
			\begin{aligned}
				&\partial_t D^k\phi + u\cdot\nabla D^k\phi + \sqrt{\gamma}\;\mathrm{div}D^k u
				= -[D^k,u\cdot\nabla]\phi - \tfrac{\gamma-1}{2}\,[D^k,\phi]\,\mathrm{div}u
				- \tfrac{\gamma-1}{2}\,\phi\, D^k\mathrm{div}u,\\
				&\partial_t D^k u + u\cdot\nabla D^k u + \sqrt{\gamma}\,\nabla D^k\phi + D^k u
				= -[D^k,u\cdot\nabla]u - \tfrac{\gamma-1}{2}\,[D^k,\phi]\nabla\phi
				- \tfrac{\gamma-1}{2}\,\phi\, D^k\nabla\phi.
			\end{aligned}
			\right.
		\end{equation}
		Taking the $L^2$-inner product of \eqref{akuk1} and $(D^{k}\phi,D^{k} u)$, and using Lemma \ref{jiaohuan}, we deduce
		\begin{align}
			&\frac{1}{2}\frac{\mathrm{d}}{\mathrm{d}t}\|(D^{k} \phi,D^{k} u)\|_{L^2}^2
			+\|D^{k} u\|_{L^2}^2 \nonumber\\
			&=\int_{\mathbb{R}^3} D^k \phi \cdot \big(-[D^k,u\cdot\nabla]\phi - u\cdot\nabla D^k\phi- \tfrac{\gamma-1}{2}\,[D^k,\phi]\,\mathrm{div}u
				- \tfrac{\gamma-1}{2}\,\phi\, D^k\mathrm{div}u\big)\mathrm{d}x \nonumber\\
                &\quad + \int_{\mathbb{R}^3} D^k u \cdot \big(-[D^k,u\cdot\nabla]u-u\cdot\nabla D^k u- \tfrac{\gamma-1}{2}\,[D^k,\phi]\nabla\phi
				- \tfrac{\gamma-1}{2}\,\phi\, D^k\nabla\phi\big)\mathrm{d}x \nonumber\\
			&\leq  C\big(\|\nabla u\|_{L^\infty} \|D^k \phi\|_{L^2}^2+\|\nabla \phi\|_{L^\infty}\|D^k\phi\|_{L^2}\|D^k u\|_{L^2}\big)+C\|\nabla u\|_{L^\infty}\|D^k\phi\|_{L^2}^2 \nonumber\\
            &\quad+C (\gamma-1)\big(\|\nabla \phi\|_{L^\infty} \|D^k\phi\|_{L^2}\|D^k u\|_{L^2}+\|\nabla u\|_{L^\infty}\|D^k\phi\|_{L^2}^2\big) \nonumber\\
            &\quad+C(\gamma-1)(\|\nabla \phi\|_{L^\infty} \|D^k\phi\|_{L^2}\|D^k u\|_{L^2}+\|\nabla u\|_{L^\infty}\|D^ku\|_{L^2}^2)\nonumber\\
            &\quad-\frac{\gamma-1}{2}\int_{\mathbb{R}^3}\phi \big(D^k\mathrm{div}u D^k\phi+ D^k\nabla\phi\cdot D^ku\big)\mathrm{d}x  \nonumber\\
			&\leq C\|\nabla u\|_{L^\infty}\big( \|D^k \phi\|_{L^2}^2+\|D^ku\|_{L^2}^2\big)
             \nonumber\\
             &\quad+\hat{C}(\gamma-1)(\|\nabla \phi\|_{L^{\infty}}+\|\nabla u\|_{L^{\infty}})\big( \|D^k \phi\|_{L^2}^2+\|D^ku\|_{L^2}^2\big),
			\label{aualpha1}
		 \end{align} 
		where we have used the fact that
			\begin{align}
				\int_{\mathbb{R}^3}\phi \big(D^k(\mathrm{div} u) D^k\phi+ D^k\nabla\phi\cdot D^ku\big)\mathrm{d}x
				&=\int_{\mathbb{R}^3}\phi \,\mathrm{div}(D^k\phi  D^ku)\mathrm{d}x\nonumber\\
				&=-\int_{\mathbb{R}^3}\nabla \phi \cdot (D^k\phi  D^ku)\mathrm{d}x\nonumber\\
				&\leq C\|\nabla\phi\|_{L^\infty}
                \big(\|D^k\phi\|_{L^2}^2+\|D^ku\|_{L^2}^2\big).\label{remind}
			\end{align}
		Summing the above inequality with respect to $k$  from $0$ to $3$ gives
		\begin{equation*}
			\frac{1}{2}\frac{\mathrm{d}}{\mathrm{d}t}\|(\phi,u)\|_{H^3}^2
			+\|u\|_{H^3}^2\leq (C+(\gamma-1)\hat{C})(\|\nabla \phi\|_{L^{\infty}}+\|\nabla u\|_{L^{\infty}})\|(\phi,u)\|_{H^3}^2.
		\end{equation*}
		Using the Gr\"{o}nwall inequality and 
        assumption \eqref{priori}, we obtain
		\begin{equation*}
			\sup\limits_{t\in[0,T]}\|(\phi,u)(t)\|_{H^3}^2+\int_0^T\|u(t)\|_{H^3}^2 \,\mathrm{d}t
			\leq e^{C\delta+(\gamma-1)\hat{C}\delta}\|(\phi_0,u_0)\|_{H^3}^2,
		\end{equation*}
     which together with the smallness of $\delta$ yields \eqref{aH3}.
     
\smallskip    
{\bf 3. Equivalence between $\phi$ and $\rho-1$.}
For simplicity, we write $\tilde{C}=e^{C+\hat{C}(\gamma-1)}$ for brevity. The explicit dependence on $\gamma-1$ plays a key role in proving the isentropic limit of the solutions. 
We first prove that, under assumption \eqref{priori}, 
\begin{align}
			&\sup\limits_{t\in[0,T]}\|(\rho-1,u)(t)\|^2_{H^3}+\int_0^T\|u(t)\|^2_{H^3} \,\mathrm{d}t \leq  \tilde{C}\|(\rho_0-1,u_0)\|^2_{H^3},\label{rhouH3}\\[1mm]
	&\sup\limits_{t\in[0,T]}\|D^3(\rho-1,u)(t)\|_{L^2}\leq  \tilde{C}\delta_0.\label{d3rho}
\end{align}

In fact, we need to prove
\begin{equation}\label{equiv-phi-rho}
    \|D^k \phi\|_{L^2} \sim \|D^k(\rho-1)\|_{L^2}
    \qquad \text{for ~\,$k=0,1,2,3$}.
\end{equation}
Recall that
\begin{equation}\label{phi-def}
    \phi=\frac{2\sqrt{\gamma}}{\gamma-1}
    \bigl(\rho^{\frac{\gamma-1}{2}}-1\bigr).
\end{equation}

For \(k=0\), by the mean value theorem, there exists \(\tilde\rho\) between \(1\) and \(\rho\) such that
\begin{equation}\label{phi-rho-mvt}
    \phi
    =\sqrt{\gamma}\,\tilde\rho^{\frac{\gamma-3}{2}}(\rho-1).
\end{equation}
Since $\rho\in
    \bigl(e^{-\delta}\underline\rho,e^{\delta}\overline\rho\bigr)$, then \(\tilde\rho\) is uniformly bounded from above and below. It follows immediately that
\begin{equation}\label{phi-rho-l2}
    \|\phi\|_{L^2}\sim \|\rho-1\|_{L^2}.
\end{equation}

For the first-order derivatives, differentiating \eqref{phi-def}, we obtain
\begin{equation}\label{grad-phi}
    D\phi=\sqrt{\gamma}\,\rho^{\frac{\gamma-3}{2}}D\rho.
\end{equation}
Again, by the uniform upper and lower bounds of \(\rho\), we deduce
\begin{equation}\label{d1-equiv}
    \|D\phi\|_{L^2}\sim \|D\rho\|_{L^2}.
\end{equation}

We now turn to the second-order derivatives. Differentiating \eqref{grad-phi} once more yields
\begin{equation}\label{rho-second}
   D^2\rho
=\frac{1}{\sqrt{\gamma}}\,\rho^{\frac{3-\gamma}{2}}D^2\phi
+\frac{3-\gamma}{2\sqrt{\gamma}}\,
\rho^{\frac{\gamma-5}{2}} (D\rho)^2,
\end{equation}
so that
\begin{equation}\label{rho-second-est-1}
    \|D^2\rho\|_{L^2}
    \le C\big(\|D^2\phi\|_{L^2}
    +\|D \rho\|_{L^4}^2\big).
\end{equation}
Similarly, from \eqref{rho-second},
\begin{equation}\label{phi-second-est-1}
    \|D^2\phi\|_{L^2}
    \le C\big(\|D^2\rho\|_{L^2}
    +\|D\rho\|^2_{L^4}\big).
\end{equation}

It remains to control $\|D\rho\|_{L^4}^2$. By \eqref{grad-phi}, the Gagliardo--Nirenberg inequality, and \eqref{aH3},
\begin{equation} \label{quad-est-phi}
      \|D\rho\|_{L^4}^2 \le C\|D\phi\|_{L^4}^2 \le C\|D\phi\|_{L^2}^{\frac{1}{2}}\|D^2\phi\|_{L^2}^{\frac{3}{2}}
    \le \widetilde C\|D^2\phi\|_{L^2}.
\end{equation}
On the other hand,
\begin{equation}\label{quad-est-rho}
     \|D\rho\|_{L^4}^2 \le C\|D\rho\|_{L^2}^{\frac{1}{2}}\|D\rho\|_{L^6}^{\frac{3}{2}} \le C\|D\phi\|_{L^2}^{\frac{1}{2}}\|D^2\rho\|_{L^2}^{\frac{3}{2}}
    \le \widetilde C\|D^2\rho\|_{L^2}.
\end{equation}
    
Substituting \eqref{quad-est-phi}--\eqref{quad-est-rho} into
\eqref{rho-second-est-1}--\eqref{phi-second-est-1}, we have
\begin{equation}\label{d2-equiv}
    \|D^2\phi\|_{L^2}\sim \|D^2\rho\|_{L^2}.
\end{equation}

For the third-order derivatives, one differentiates \eqref{grad-phi} once more. The resulting expression consists of the leading linear term \( \rho^{\frac{\gamma-3}{2}}D^3\rho \) and the lower-order nonlinear terms involving the products of derivatives of \(\rho\) up to second order. Using the already established bounds \eqref{d1-equiv} and \eqref{d2-equiv}, together with the uniform bounds on \(\rho\), these lower-order terms can be controlled in the same way as above. Therefore, we have
\begin{equation}\label{d3-equiv}
    \|D^3\phi\|_{L^2}\sim \|D^3\rho\|_{L^2}.
\end{equation}
Combining \eqref{phi-rho-l2}, \eqref{d1-equiv}, and \eqref{d2-equiv}--\eqref{d3-equiv}, we obtain \eqref{equiv-phi-rho}.

Next, recalling the initial assumption in \eqref{H3},
\begin{equation}\label{h3-assumption-proof}
    \|D^3(\rho_0-\rho_*,u_0)\|_{L^2}\le \delta_0,
\end{equation}
it follows from \eqref{equiv-phi-rho} that
\begin{equation}\label{d3phi-small}
    \sup_{t\in[0,T]}\|D^3(\phi,u)(t)\|_{L^2}\le \tilde{C}\delta_0.
\end{equation}
Together with \eqref{priori}, this yields \eqref{d3rho}.

Finally, using \(M_0\le \delta_0^{-1/11}\), the Gagliardo--Nirenberg inequality gives
\begin{align}
    &\|Du\|_{L^2}
    \le C\|u\|_{L^2}^{\frac{2}{3}}\|D^3u\|_{L^2}^{\frac{1}{3}} 
    \le \tilde{C} M_0^{\frac{2}{3}}\delta_0^{\frac{1}{3}}
    \le \tilde{C}\delta_0^{\frac{3}{11}}, \label{010101}\\
 &\|D^2u\|_{L^2}
    \le C\|u\|_{L^2}^{\frac{1}{3}}\|D^3u\|_{L^2}^{\frac{2}{3}} 
    \le \tilde{C} M_0^{\frac{1}{3}}\delta_0^{\frac{2}{3}}
    \le \tilde{C}\delta_0^{\frac{7}{11}}. \label{010102}
\end{align}

\rm{Theorem} \ref{thm3.1} follows directly from \eqref{rhouH3}.       
	\end{proof}
	
	\subsection{Decay Rates of the Solution}\label{sec.4}
	In this section, we establish the optimal decay rates of the solution and then employ these to prove Theorem \ref{main}.
	
	Let $U=(\phi,u)^{\top}$ and $U_0=(\phi_0,u_0)^{\top}$. We rewrite system \eqref{au1}--\eqref{initial-au} as the following vector form
	\begin{equation}\label{L-system-M}
		\left\{
		\begin{aligned}
			& \partial_t U +\mathcal{L}U = F,\\
			& U(0,x) = U_0,
		\end{aligned}
		\right.
	\end{equation}
	where  $\mathcal{L}$ and $F$ are given by
	\begin{equation*}
		\mathcal{L} = \left( {\begin{array}{*{20}{c}}
				{0}&{{\sqrt{\gamma}\mathrm{div}}}\\
				{\sqrt{\gamma}\nabla}&{I}
		\end{array}} \right),
		\qquad \quad
		F=\left(
		\begin{array}{*{20}{c}}
			{-u \cdot \nabla \phi-\frac{\gamma-1}{2} \phi  {\mathrm{div}} u}\\
			{-u \cdot \nabla u-\frac{\gamma-1}{2} \phi\nabla \phi}
		\end{array} \right)
		:=\left(
		\begin{array}{*{20}{c}}
			{f_1}\\
			{f_2}
		\end{array} \right).
	\end{equation*}
	The solution  of \eqref{L-system-M} can be expressed as
	\begin{equation}\label{duhamel}
		U=G\ast U_0 +\int^t_0 G(t-\tau)\ast F(\tau)\mathrm{d}\tau,
	\end{equation}
	where \begin{equation*}
		G(t,x) = \left(
		{\begin{array}{*{20}{c}}
				{G_{11}}&{G_{12}}\\
				{G_{21}}&{G_{22}}
		\end{array}} \right)
	\end{equation*}
	denotes the Green function of the linear operator $\mathcal{L}$.\;Although the Green function have been investigated in \cite{STW}, we still present some additional properties of the Green function and the process of proof for completeness.
	
	The Green function $G$ satisfies the following linear system:
	\begin{equation}\label{xianxing}
		\left\{
		\begin{aligned}
			& \partial_t G +\mathcal{L}G = 0,\\
			& G(0,x) = \delta(x)I,
		\end{aligned}
		\right.
	\end{equation}
	where $\delta(x)$ is the standard Dirac delta function. Taking the Fourier transform of the above system yields
	\begin{equation}
		\left\{
		\begin{aligned}
			& \partial_t \hat{G} +\hat{\mathcal{L}}\hat{G} = 0,\\
			& \hat{G}(0,\xi) = I,
		\end{aligned}
		\right.
	\end{equation}
	where
	\begin{equation*}
		\hat{\mathcal{L}}_\xi = \left( {\begin{array}{*{20}{c}}
				{0}&{i \sqrt{\gamma}\xi^{\top}}\\
				{i \sqrt{\gamma}\xi}&{I}
		\end{array}} \right).
	\end{equation*}
	We can calculate the eigenvalues of $\hat{\mathcal{L}}$ from
	\begin{equation*}
		|\lambda I+\hat{\mathcal{L}}_\xi|{\rm{ = }}\left| {\begin{array}{*{20}{c}}
				{\lambda}&{i\sqrt{\gamma}\xi}\\
				{i\sqrt{\gamma}\xi}&({\lambda+1})I\\
		\end{array}} \right|
		=(\lambda+1)^2(\lambda^2+\lambda+\gamma|\xi|^2)=0,
	\end{equation*}
	to obtain
	\begin{equation*}
		\lambda_1=\lambda_2=-1,
		\quad \lambda_3=-\frac{1-\sqrt{1-4\gamma|\xi|^2}}{2}, \qquad 
		\lambda_4=-\frac{1+\sqrt{1-4\gamma|\xi|^2}}{2}.
	\end{equation*}
	For the eigenvalues $\lambda_3$ and $\lambda_4$, we have
	\begin{lem}\label{lem-tezheng}
		Let $r_0$ be a small positive constant satisfying $r_0<\frac{1}{2\sqrt{\gamma}}<\frac{1}{2}$.
		\begin{itemize}
			\item When $|\xi|< r_0$, the eigenvalues $\lambda_3$ and $\lambda_4$ satisfy
			\begin{equation}\label{tezheng-l}
				\lambda_3=-\gamma|\xi|^2+O(|\xi|^4),
				\quad \lambda_4=-1+\gamma|\xi|^2+O(|\xi|^4).
			\end{equation}
			\item When $|\xi|\geq r_0$, it holds that 
			\begin{equation}\label{tezheng-mh}
				{\rm{Re}}\lambda_i\leq -r_0^2 \qquad \text{for $i=3, 4.$}
			\end{equation}
		\end{itemize}
	\end{lem}
	\begin{proof}
		For $|\xi|< r_0$, \eqref{tezheng-l} is directly obtained by using the Taylor expansion formula.
		For $|\xi|> \frac{1}{2}$, $\lambda_3$ and $\lambda_4$ are complex roots and
		\begin{equation*}
			{\rm{Re}}\lambda_3={\rm{Re}}\lambda_4=-\frac{1}{2}.
		\end{equation*}
		For $r_0\leq|\xi|\leq \frac{1}{2}$, $\lambda_3$ and $\lambda_4$ are real roots and satisfy
		\begin{align*}
			\lambda_3&=-\frac{1-\sqrt{1-4\gamma|\xi|^2}}{2}
			=-\frac{2\gamma|\xi|^2}{1+\sqrt{1-4\gamma|\xi|^2}}
			\leq -\gamma|\xi|^2
			\leq -r_0^2,\\
			\lambda_4&=-\frac{1+\sqrt{1-4\gamma|\xi|^2}}{2}
			\leq -\frac{1}{2}.
		\end{align*}
		Thus, the proof is completed.
	\end{proof}
	
	Based on the above eigenvalue estimates, we have the following estimates of the Green function.
	\begin{lem}\label{gl-lam}
		For any smooth function $\zeta$, the following estimates hold for any integer $k\geq 0$,
        \begin{equation}\label{G}
\|D^k G_{ij}\ast \zeta\|_{L^2} \le 
\begin{cases}
C(1+t)^{-\frac{k}{2}}\big(\|\zeta\|_{L^2}
			+\|D^k\zeta\|_{L^2}\big) & \text{for \;$i+j=2$},\\[1mm]
C(1+t)^{-\frac{k+1}{2}}\big(\|\zeta\|_{L^2}
			+\|D^k\zeta\|_{L^2}\big)\quad & \text{for \;$i+j=3$},\\
C(1+t)^{-\frac{k+2}{2}}\big(\|\zeta\|_{L^2}
			+\|D^k\zeta\|_{L^2}\big)\quad & \text{for \;$i+j=4$},
\end{cases}
\end{equation}
     where $i, j\in \{1,2\}$ and the above $C$ depends on $\gamma$ but does not have a singularity as $\gamma \to 1$. The low-frequency part of $G_{ij}$ satisfies
		\begin{equation}\label{GL}
\|D^k G_{ij}^{\ell}\ast \zeta\|_{L^2} \le 
\begin{cases}
C(1+t)^{-\frac{k}{2}}\|\zeta\|_{L^2} & \text{for \;$i+j=2$},\\[1mm]
C(1+t)^{-\frac{k+1}{2}}\|\zeta\|_{L^2}\quad & \text{for \;$i+j=3$},\\
C(1+t)^{-\frac{k+2}{2}}\|\zeta\|_{L^2}\quad & \text{for \;$i+j=4$}.
\end{cases}
\end{equation}
	\end{lem}
	
	\begin{proof}
		A direct computation gives the  exact expression of the Fourier transform of $G$ as follows:
		\begin{align}\label{01001}
			\hat{G}(t,\xi)
			= \left( {\begin{array}{*{20}{c}}
					{\hat{G}_{11}}&{\hat{G}_{12}}\\
					{\hat{G}_{21}}&{\hat{G}_{22}}
			\end{array}} \right)
			=\left( {\begin{array}{*{20}{c}}
					{\frac{\lambda_3 e^{\lambda_4t}
							-\lambda_4e^{\lambda_3t}}{\lambda_3-\lambda_4}}
					&{-i\xi^{\top}\frac{e^{\lambda_3 t}
							-e^{\lambda_4 t}}{\lambda_3-\lambda_4}}\\
					{-i\xi\frac{e^{\lambda_3 t}
							-e^{\lambda_4t}}{\lambda_3-\lambda_4}}
					&{e^{-t}I
						+\big(\frac{\lambda_3 e^{\lambda_3 t}
							-\lambda_4e^{\lambda_4t}}{\lambda_3-\lambda_4}
						-e^{-t}\big)\frac{\xi\xi^{\top}}{|\xi|^2}}
			\end{array}} \right).
		\end{align}
		For any smooth function $\zeta$, we have
		\begin{align*}
			\| D^{k}G_{11}\ast \zeta\|^2_{L^2}
			&=C\|(i\xi)^{k}\hat{G}_{11}\hat{\zeta}\|^2_{L^2}\\
			&=C\Big(\int_{|\xi|<r_0}
			+\int_{|\xi|\geq r_0}\Big)|(i\xi)^{k}\hat{G}_{11}\hat{\zeta}|^2\mathrm{d}\xi\\
			&\leq C\int_{|\xi|<r_0} |(i\xi)^{k}e^{-|\xi|^2t}\hat{\zeta}|^2\mathrm{d}\xi
			+C\int_{|\xi|\geq r_0} |(i\xi)^{k}e^{-r^2_0t}\hat{\zeta}|^2\mathrm{d}\xi\\
			&\leq C\|\widehat{D^k H}\hat{\zeta}\|_{L^2}^2
			+C\|e^{-r^2_0t}\widehat{\delta(x)}\widehat{D^{k} \zeta}\|^2_{L^2}\\
			&\leq C\|D^k H\|_{L^1}^2\|\zeta\|_{L^2}^2
			+C e^{-2r^2_0t}\|D^k\zeta\|_{L^2}^2   \\
			&\leq C(1+t)^{-k}\big(\|\zeta\|_{L^2}^2+\|D^k\zeta\|_{L^2}^2\big),
		\end{align*}
		where $H(t,x)=(4\pi{t})^{-\frac{3}{2}}e^{-\frac{|x|^2}{4t}}$ is the heat kernel and $\delta(x)$ is a Dirac delta function. It should be emphasized that all constants $C$ appearing in the above estimates may depend on the adiabatic exponent $\gamma$, but are uniformly bounded with respect to $\gamma\in (1, 3]$.

		This uniformity follows from the fact that
		\[
		\lambda_3-\lambda_4=\sqrt{1-4\gamma|\xi|^2}
		\]
		remains strictly positive for all $|\xi|<r_0$, where $r_0>0$ is chosen, independent of $\gamma$.  Moreover, if $|\xi|\rightarrow \frac{1}{2\sqrt{\gamma}}$, then $\lambda_3\rightarrow \lambda_4$ so that 
		\begin{align*}
		&\lim_{\lambda_3\rightarrow \lambda_4}	{\frac{\lambda_3 e^{\lambda_4t}
					-\lambda_4e^{\lambda_3t}}{\lambda_3-\lambda_4}}=(1-\lambda_4t)e^{\lambda_4t}, \\
                   &\lim_{\lambda_3\rightarrow \lambda_4} \frac{e^{\lambda_3 t}
					-e^{\lambda_4t}}{\lambda_3-\lambda_4}=te^{\lambda_4t},\\
				&\lim_{\lambda_3\rightarrow \lambda_4}	\frac{\lambda_3 e^{\lambda_3 t}
						-\lambda_4e^{\lambda_4t}}{\lambda_3-\lambda_4}=(1+\lambda_4t)e^{\lambda_4t}.
		\end{align*}	
		Then it is deduced that, when $|\xi|\rightarrow \frac{1}{2\sqrt{\gamma}}$, the Green function has no singularity point, which implies that it is bounded.
	
		Thus we obtain
		\begin{align*}
			\| D^{k}G_{11}\ast \zeta\|_{L^2}\leq C(1+t)^{-\frac{k}{2}}\big(\|\zeta\|_{L^2}+\|D^k\zeta\|_{L^2}\big).
		\end{align*}
		Furthermore, $\eqref{G}_2$ can be estimated in a similar way as above, here we omit the details.
		Moreover, according to the definition of the low-frequency projection in  \eqref{p1pinf}, we have
		\begin{align*}
			\| D^k G_{11}^{\ell} \ast \zeta\|^2_{L^2}
			&=C\|(i\xi)^{k}\hat{G}_{11}^{\ell}\hat{\zeta}\|^2_{L^2}\\
			&=C\int_{|\xi|<r_0}|(i\xi)^{k}\hat{G}_{11}\hat{\zeta}|^2\mathrm{d}\xi\\
			&\leq C\int_{|\xi|<r_0} |(i\xi)^{k}e^{-|\xi|^2t}\hat{\zeta}|^2\mathrm{d}\xi\\
			&\leq C\|\widehat{D^k H}\hat{\zeta}\|_{L^2}^2\\
			&\leq C\|D^k H\|_{L^1}^2\|\zeta\|_{L^2}^2\\
			&\leq C(1+t)^{-k}\|\zeta\|_{L^2}^2,
		\end{align*}
		which yields $\eqref{GL}_{1}$.  The remaining inequalities in $\eqref{GL}_{2}$ can be proved similarly. This completes the proof.
	\end{proof}

	Before analyzing the decay rates of the solution, we need to establish the damping-type estimates of $(\phi,u)$.
	
	\begin{lem}\label{d12-lem1}
		Under assumption \eqref{priori}, the following estimates hold\,\rm{:}
		\begin{equation}\label{dadu-1}
			\|(D \phi,D u)\|_{H^2}^2\leq Ce^{-C\delta_0^{\frac{9}{11}}t}\|(D \phi_0,D u_0)\|_{H^2}^2
		+\tilde{C}\delta_0^{\frac{9}{11}}\int_0^te^{-C\delta_0^{\frac{9}{11}}(t-\tau)} \|D \phi^{\ell}(\tau)\|_{L^2}^2\mathrm{d}\tau,
		\end{equation}
		\begin{equation}\label{dadu-2}
			\|(D^2 \phi,D^2 u)\|_{H^1}^2\leq Ce^{-C\delta_0^{\frac{9}{11}}t}	\|(D^2 \phi_0,D^2 u_0)\|_{H^1}^2
		+\tilde{C}\delta_0^{\frac{9}{11}}\int_0^te^{-C\delta_0^{\frac{9}{11}}(t-\tau)} \|D^2 \phi^{\ell}(\tau)\|_{L^2}^2\mathrm{d}\tau.
		\end{equation}
	\end{lem}
	
	\begin{proof}
		Taking the $L^2$-inner product of \eqref{au1}$_2$ with $\nabla \phi$, we have
		\begin{align*}
			&\sqrt{\gamma}\|\nabla \phi\|_{L^2}^2
			+ \frac{\mathrm{d}}{\mathrm{d}t} \int_{\mathbb{R}^3} u\cdot\nabla \phi \, \mathrm{d}x\\
			&= \int_{\mathbb{R}^3} u\cdot\nabla \partial_t\phi\, \mathrm{d}x
			- \int_{\mathbb{R}^3} \nabla \phi \cdot (u\cdot \nabla u) \, \mathrm{d}x
			-\int_{\mathbb{R}^3} \nabla \phi \cdot u \, \mathrm{d}x-\frac{\gamma-1}{2}\int_{\mathbb{R}^3}\phi\nabla\phi\cdot\nabla\phi \, \mathrm{d}x
			\\
			&= \frac{\gamma-1}{2}\int_{\mathbb{R}^3}\phi (\mathrm{div}u)^2  \, \mathrm{d}x+\sqrt{\gamma}\int_{\mathbb{R}^3} (\mathrm{div}u)^2 \, \mathrm{d}x+\int_{\mathbb{R}^3} \mathrm{div}\, u (u\cdot\nabla \phi)
			\, \mathrm{d}x \\
			&\quad\,- \int_{\mathbb{R}^3} \nabla \phi \cdot (u\cdot \nabla u) \, \mathrm{d}x
			-\int_{\mathbb{R}^3} \nabla \phi \cdot u \, \mathrm{d}x-\frac{\gamma-1}{2}\int_{\mathbb{R}^3}\phi\nabla\phi\cdot\nabla\phi \, \mathrm{d}x \\
			&\leq  \frac{\gamma-1}{2} \|\phi \|_{L^{\infty}}\big(\|\nabla u\|_{L^2}^2+\|\nabla \phi\|_{L^2}^2\big)
			+\sqrt{\gamma} \|\nabla u\|_{L^2}^2 \\
			&\quad\,+\| \nabla u \|_{L^{\infty}}(\|u\|_{L^2}^2
            +\|\nabla \phi\|_{L^2}^2)+ \frac{1}{4}\|\nabla \phi\|_{L^2}^2
			+\| u\|_{L^2}^2\\
			&\leq  \big(
			\frac{\gamma-1}{2}\|\phi\|_{L^\infty}
			+ \|\nabla u\|_{L^\infty}
			+ \frac14
			\big)\|\nabla \phi\|_{L^2}^2  \\
			&\quad\, 
            +\big(\sqrt{\gamma}+\frac{\gamma-1}{2}\|\phi\|_{L^\infty}
			\big)\|\nabla u\|_{L^2}^2 + \big(1+\|\nabla u\|_{L^\infty}\big)\|u\|_{L^2}^2.
		\end{align*}
		Applying the estimates that 
        $\| \phi \|_{L^{\infty}}\leq C\|D \phi\|_{L^2}^{\frac{1}{2}}\|D^2 \phi\|_{L^2}^{\frac{1}{2}}\leq \tilde{C}\delta_0^{\frac{5}{11}}$
        and $\|\nabla u\|_{L^\infty}\leq \tilde{C}\delta_0^{\frac{9}{11}}$
        and using \eqref{010101}--\eqref{010102}, together with the smallness of $\delta_0$ and $\gamma>1$, we have
		\begin{equation*}
			\frac{1}{2}\|\nabla \phi\|_{L^2}^2
			+\frac{\mathrm{d}}{\mathrm{d}t} \int_{\mathbb{R}^3}u\cdot\nabla \phi \, \mathrm{d}x
			\leq \tilde{C}_3\|u\|_{H^1}^2.
		\end{equation*}
		It should be mentioned that constant $\tilde{C}_3$ here depends on $\gamma$ but is uniformly bounded in  
        $\gamma \in (1,3]$.
		Applying $D^{k}$, $k=1,2$, to \eqref{au1}$_2$, taking the $L^2$-inner product of the resulting equation and $\nabla D^{k} \phi$, and following the above process, we deduce
		\begin{align}
			&\frac{1}{2}\|\nabla D \phi\|_{L^2}^2
			+\frac{\mathrm{d}}{\mathrm{d}t} \int_{\mathbb{R}^3}
			D u\cdot\nabla D \phi \, \mathrm{d}x\leq \tilde{C}_4\|D u\|_{H^1}^2,\label{nabla2a1}\\
		&\frac{1}{2}\|\nabla D^2 \phi\|_{L^2}^2+
			\frac{\mathrm{d}}{\mathrm{d}t} \int_{\mathbb{R}^3}
			D^2 u\cdot\nabla D^2 \phi \, \mathrm{d}x\leq \tilde{C}_5\|D^2 u\|_{H^1}^2,\label{nabla3a1}
		\end{align}
		where $\tilde{C}_3$, $\tilde{C}_4$, and $\tilde{C}_5$ are positive constants independent of time. They may depend on $\gamma$, but remain uniformly bounded in $\gamma\in(1,3]$. 
        Moreover, summing \eqref{aualpha1} with respect to $k$ from 1 to 3, 
        we can also obtain
		\begin{align}\label{nabla3u1}
			\frac{1}{2}\frac{\mathrm{d}}{\mathrm{d}t} \|(D \phi,D u)\|_{H^2}^2
			+\|D u\|_{H^2}^2
			\leq \tilde{C} \big(\|\nabla \phi\|_{L^{\infty}}
			+\|\nabla u \|_{L^{\infty}}\big)
			\|(D \phi,D u)\|_{H^2}^2.
		\end{align}
		Multiplying \eqref{nabla2a1} and \eqref{nabla3a1} by $\beta_1>0$
        and adding the resulting equations to \eqref{nabla3u1} yield
		\begin{align}
			&\frac{\mathrm{d}}{\mathrm{d}t} \Big(
			\frac{1}{2}\|(D \phi,D u)\|_{H^2}^2
			+\beta_1 \int_{\mathbb{R}^3}D u\cdot\nabla D \phi \, \mathrm{d}x
			+ \beta_1\int_{\mathbb{R}^3}D^2 u\cdot\nabla D^2 \phi \, \mathrm{d}x\Big)
			+ \frac{1}{2} \beta_1\|\nabla D \phi\|_{H^1}^2
			+ \|D u\|_{H^2}^2 \notag\\
			&\leq \tilde{C} \big(\|\nabla \phi\|_{L^{\infty}}
			+\|\nabla u \|_{L^{\infty}}\big)
			\|(D \phi,D u)\|_{H^2}^2
			+\beta_1(\tilde{C}_4+\tilde{C}_5)\|D u\|_{H^2}^2\notag\\
			&\leq \tilde{C} (\|D^2 \phi\|_{L^2}^{\frac{1}{2}}
			\|D^3 \phi\|_{L^2}^{\frac{1}{2}}
			+\|D^2 u\|_{L^2}^{\frac{1}{2}}
			\|D^3 u\|_{L^2}^{\frac{1}{2}})
			\|(D \phi,D u)\|_{H^2}^2+\beta_1(\tilde{C}_4+\tilde{C}_5)\|D u\|_{H^2}^2\notag\\
			&\leq \tilde{C} \delta_0^{\frac{9}{11}}\|(D \phi,D u)\|_{H^2}^2
			+\beta_1(\tilde{C}_4+\tilde{C}_5)\|D u\|_{H^2}^2,\label{diff-a1}
		\end{align}
		where we have used \eqref{d3rho} in the last inequality.
        
	Let 
	\begin{align}
	A_1(t)=	\frac{1}{2}\|(D \phi,D u)\|_{H^2}^2
	+\beta_1 \int_{\mathbb{R}^3}D u\cdot\nabla D \phi \, \mathrm{d}x
	+ \beta_1\int_{\mathbb{R}^3}D^2 u\cdot\nabla D^2 \phi \, \mathrm{d}x.
	\end{align}	
	Choosing $\beta_1$ suitably small so that 
        \begin{align}\label{AD1}
	A_1(t)\sim 	\|(D \phi,D u)\|_{H^2}^2
	\end{align}	 and 
     $\beta_1(\tilde{C}_4+\tilde{C}_5)
		\leq\frac{1}{2}$,  then we obtain from \eqref{aualpha1} that
		\begin{equation*}
			\frac{\mathrm{d}}{\mathrm{d}t}
		A_1(t)
			+\frac{1}{2}\beta_1 \|\nabla D \phi\|_{H^1}^2
			+ \frac{1}{2}\|D u\|_{H^2}^2\\
			\leq \tilde{C} \delta_0^{\frac{9}{11}}\|(D \phi,D u)\|_{H^2}^2.
		\end{equation*}
		Adding $\delta_0^{\frac{9}{11}}\|D \phi\|_{L^2}^2$ to both sides of the above inequality and noting \eqref{gaodi} yields
		\begin{align*}
			&\frac{\mathrm{d}}{\mathrm{d}t}
			A_1(t)
			+\frac{1}{2}\beta_1 \|\nabla D \phi\|_{H^1}^2
			+ \frac{1}{2}\|D u\|_{H^2}^2
			+\delta_0^{\frac{9}{11}}\|D \phi\|_{L^2}^2\\
			&\leq \tilde{C} \delta_0^{\frac{9}{11}}
			\big(\|D \phi^{\ell}\|_{L^2}^2+\|D \phi^{h} \|_{L^2}^2
			+\|D u\|_{H^2}^2\big)\\
			& \leq \tilde{C} \delta_0^{\frac{9}{11}} \big(\|D \phi^{\ell} \|_{L^2}^2
			+ \|D^{2} \phi \|_{L^2}^2 +\|D u\|_{H^2}^2\big).
		\end{align*}
		Choosing $\delta_0$ is sufficiently small so that $\tilde{C}\delta_0^{\frac{9}{11}}
		< \frac{1}{4}\mathrm{min}\{1, \beta_1\}$, we obtain 
		\begin{align}
			\frac{\mathrm{d}}{\mathrm{d}t}A_1(t)
			+C\delta_0^{\frac{9}{11}}\|(D \phi,D u)\|_{H^2}^2
			\leq \tilde{C}\delta_0^{\frac{9}{11}}\|D \phi^{\ell}\|_{L^2}^2.
		\end{align}	
	With the help of \eqref{AD1}, we solve the above differential inequality to obtain 
	\begin{align}
		\|(D \phi,D u)\|_{H^2}^2\leq Ce^{-C\delta_0^{\frac{9}{11}}t}\|(D \phi_0,D u_0)\|_{H^2}^2
		+\tilde{C}\delta_0^{\frac{9}{11}}\int_0^te^{-C\delta_0^{\frac{9}{11}}(t-\tau)} \|D \phi^{\ell}(\tau)\|_{L^2}^2\mathrm{d}\tau.
	\end{align}	 
	
		Furthermore, \eqref{dadu-2} can be derived in a similar manner. Here, we omit the details for the sake of simplicity.
	\end{proof}	
	
	The damping-type estimate of the third-order derivative of $(\phi,u)$ is obtained in the following lemma.
	It should be pointed out that, since \eqref{au1} is not of 
    dissipative structure, which means there is no higher regularity of $(\phi,u)$,
	the estimate in the following lemma is based on the high-frequency part of \eqref{au1}:
	\begin{equation}\label{high-f}
		\left\{
		\begin{aligned}
			& \partial_t \phi^h + \sqrt{\gamma}\,\mathrm{div}u^h= f^h_1,\\
			& \partial_t u^h+\sqrt{\gamma}\,\nabla \phi^h+u^h =f^h_2,
		\end{aligned}
		\right.
	\end{equation}
	where $f_1=-u \cdot \nabla \phi-\frac{\gamma-1}{2} \phi \, \mathrm{div} u$ and 
    $f_2=-u \cdot \nabla u-\frac{\gamma-1}{2} \phi\nabla \phi.$
	\begin{lem}\label{d3-lem1}
		Under assumption \eqref{priori}, 
		\begin{align}\label{dadu-3}
			\|(D^3 \phi, D^3 u)\|_{L^2}^2\leq C e^{-C\delta_0^{\frac{9}{11}}t}\|(D^3\phi_0, D^3u_0)\|_{L^2}^2+\tilde{C}\delta_0^{\frac{9}{11}}\int_0^t e^{-C\delta_0^{\frac{9}{11}}(t-\tau)}\|D^3\phi^\ell(\tau)\|_{L^2}^2\mathrm{d}\tau.	
		\end{align}	
	\end{lem}	
	\begin{proof}
		Applying $D^{2}$  to \eqref{high-f}$_2$ and then taking the $L^2$-inner product of the resulting equation with $\nabla D^{2} \phi^{h}$, we deduce
        \begin{align}
			&\sqrt{\gamma}\|\nabla D^2 \phi^h\|_{L^2}^2
			+\frac{\mathrm{d}}{\mathrm{d}t} \int_{\mathbb{R}^3} D^2 u^h\cdot \nabla D^2 \phi^h \, \mathrm{d}x \nonumber\\
			&=\int_{\mathbb{R}^3} D^2 u^h\cdot \nabla D^2 \partial_t\phi^h \, \mathrm{d}x
			-\int_{\mathbb{R}^3} D^2 u^h\cdot \nabla D^2 \phi^h \, \mathrm{d}x
			+\int_{\mathbb{R}^3} D^2 f_2^h\cdot \nabla D^2 \phi^h \, \mathrm{d}x \nonumber\\
			&=\int_{\mathbb{R}^3} D^2 u^h\cdot \nabla D^2(-\sqrt{\gamma}{\mathrm{div}}u^h+ f^h_1)\, \mathrm{d}x
			-\int_{\mathbb{R}^3} D^2 u^h\cdot \nabla D^2 \phi^h \, \mathrm{d}x
			+\int_{\mathbb{R}^3} D^2 f_2^h\cdot \nabla D^2 \phi^h \, \mathrm{d}x\nonumber\\
			&=\sqrt{\gamma}\int_{\mathbb{R}^3} |D^2 {\mathrm{div}}u^h|^2\, \mathrm{d}x
			-\int_{\mathbb{R}^3} D^2 u^h\cdot \nabla D^2 \phi^h \, \mathrm{d}x +\int_{\mathbb{R}^3} D^2 u^h \cdot \nabla D^2(-u\cdot \nabla \phi)^h\, \mathrm{d}x\nonumber\\
			&\quad\quad+\int_{\mathbb{R}^3} \nabla D^2 \phi^h \cdot D^2 (-u \cdot \nabla u)^h   {\mathrm{d}}x
			-\frac{\gamma-1}{2}\int_{\mathbb{R}^3} D^2 u^h\cdot \nabla D^2 (\phi {\mathrm{div}}u)^h\, \mathrm{d}x \nonumber\\
            &\quad\quad-\frac{\gamma-1}{2}\int_{\mathbb{R}^3} \nabla D^2 \phi^h\cdot  D^2 (\phi \nabla \phi)^h\, \mathrm{d}x \nonumber\\
			&\leq \sqrt{\gamma}\|D^2 \mathrm{div}\, u^h\|_{L^2}^2
			+\epsilon \|\nabla D^2 \phi^h\|_{L^2}^2
			+C_{\epsilon} \|D^2 u^h\|_{L^2}^2
			+K_1+K_2+K_3+K_4.\label{d3ah}
		\end{align}    
For the $K_1$--term on the right-hand side of \eqref{d3ah}, we have
        \begin{align}
			K_1 &=\int_{\mathbb{R}^3}D^2 u^h \cdot \nabla D^2(-u\cdot \nabla \phi)^h\, \mathrm{d}x \nonumber\\
			& =\int_{\mathbb{R}^3}D^2 u^h \cdot \nabla D^2(-u\cdot \nabla \phi)\, \mathrm{d}x
			+\int_{\mathbb{R}^3}D^2 u^h \cdot \nabla D^2(u\cdot \nabla \phi)^{\ell}\, \mathrm{d}x \nonumber\\
			& =\int_{\mathbb{R}^3} D^2 u^h\cdot [\nabla D^2,-u\cdot\nabla]\phi \, \mathrm{d}x
			-\int_{\mathbb{R}^3} D^2 u^h\cdot (u \cdot\nabla\nabla D^2 \phi)\, \mathrm{d}x \nonumber\\
			&\quad +\int_{\mathbb{R}^3} D^2 u^h\cdot \nabla D^2(u\cdot \nabla \phi)^{\ell}\, \mathrm{d}x\nonumber\\
			&:=J_1+J_2+J_3. \label{uhf1h}
		\end{align}
For the $J_1$--term, it follows from the fact that $\|D^2 u^h\|_{L^2}
		\leq C\|D^3 u^h\|_{L^2} \leq C\|D^3 u\|_{L^2}$
		and Lemma \ref{jiaohuan}  that
		\begin{align*}
			J_1
			&\leq C\|D^3 u\|_{L^2}\big(
			\|\nabla u\|_{L^{\infty}}\|D^3 \phi\|_{L^2}
			+\|D^3 u\|_{L^2}\|\nabla \phi\|_{L^{\infty}}
			\big)\\
			&\leq C\big(\|\nabla \phi\|_{L^{\infty}}
			+\|\nabla u\|_{L^{\infty}}\big)
			\big(\|D^3 \phi\|_{L^2}^2
			+\|D^3 u\|_{L^2}^2\big).
		\end{align*}
		For the $J_2$--term, integration by parts yields
		\begin{align*}
			J_2
			&=\int_{\mathbb{R}^3} (\nabla D^2 u^h\cdot u)\cdot\nabla D^2 \phi \, \mathrm{d}x
			+\int_{\mathbb{R}^3} D^2 u^h\cdot{\mathrm{div}}u\cdot\nabla D^2 \phi \, \mathrm{d}x\\
			&\leq \|\nabla D^2 u^h\|_{L^2}\|u\|_{L^{\infty}}
			\|\nabla D^2 \phi\|_{L^2}
			+C\|\nabla u\|_{L^{\infty}}
			\|D^2 u^h\|_{L^2}
			\|\nabla D^2 \phi\|_{L^2}\\
			&\leq C\|\nabla D^2 u\|_{L^2}\|\nabla D^2 \phi\|_{L^2}
			+C\|\nabla u\|_{L^{\infty}}\big(
			\|D^3 u\|_{L^2}^2
			+\|\nabla D^2 \phi\|_{L^2}^2\big).
		\end{align*}
		From \eqref{gaodi}, it follows that $\|\nabla D^2 g^{\ell}\|_{L^2}\leq C\|\nabla D g^{\ell}\|_{L^2}
		\leq C\|\nabla D g\|_{L^2}$.
		Then, for the $J_3$--term, we obtain from Lemma \ref{jiaohuan} that
		\begin{align*}
			J_3 &\leq C \|D^3 u\|_{L^2}\|\nabla D (u\cdot\nabla \phi)\|_{L^2}\\
			&\leq C \|D^3 u\|_{L^2}\big(
			\|\nabla D u\cdot\nabla \phi\|_{L^2}
			+ \|\nabla u\cdot\nabla D \phi\|_{L^2}
			+\| u\cdot\nabla^2 D \phi\|_{L^2}\big)\\
			&\leq C \|D^3 u\|_{L^2}\big(
			\|\nabla D u\|_{L^6}\|\nabla \phi\|_{L^3}
			+ \|\nabla u\|_{L^3}\|\nabla D \phi\|_{L^6}
			+\| u\|_{L^{\infty}}\|\nabla^2 D \phi\|_{L^2}\big)\\
			&\leq C \|D^3 u\|_{L^2}\big(\|D^3 u\|_{L^2}\|\nabla \phi\|_{H^1}
			+\|\nabla u\|_{H^1}\|D^3 \phi\|_{L^2}
			+\|D^3 \phi\|_{L^2}\big)\\
			&\leq C \|D^3 \phi\|_{L^2}\|D^3 u\|_{L^2}
			+C\|D^3 u\|_{L^2}^2.
		\end{align*}
		Putting $J_i, i=1,2,3,$ into \eqref{uhf1h}, we deduce
		\begin{align}
			K_1\leq{}& C\big(\|\nabla \phi\|_{L^{\infty}}
			+\|\nabla u\|_{L^{\infty}}\big)
			\big(\|D^3 \phi\|_{L^2}^2
			+\|D^3 u\|_{L^2}^2\big) \notag\\
			&+C\|D^3 \phi\|_{L^2}\|D^3 u\|_{L^2}
			+C\|D^3 u\|_{L^2}^2.  \label{4xiang}
		\end{align}
For the $K_2$--term, similarly to the estimate of $K_1$, we decompose
\begin{align}
	K_2
	&=\int_{\mathbb R^3}\nabla D^2\phi^h\cdot D^2(-u\cdot\nabla u)^h\,\mathrm{d}x \notag\\
	&=\int_{\mathbb R^3}\nabla D^2\phi^h\cdot D^2(-u\cdot\nabla u)\,\mathrm{d}x
	-\int_{\mathbb R^3}\nabla D^2\phi^h\cdot D^2(-u\cdot\nabla u)^\ell\,\mathrm{d}x \notag\\
	&=\int_{\mathbb R^3}\nabla D^2\phi^h\cdot [D^2,-u\cdot\nabla]u\,\mathrm{d}x
	-\int_{\mathbb R^3}\nabla D^2\phi^h\cdot (u\cdot\nabla D^2u)\,\mathrm{d}x  \notag\\
	&\quad
	-\int_{\mathbb R^3}\nabla D^2\phi^h\cdot D^2(-u\cdot\nabla u)^\ell\,\mathrm{d}x \notag\\
	&:=J_4+J_5+J_6 .
	\label{K2decomp}
\end{align}
By Lemma \ref{jiaohuan}, the Sobolev inequality, and the fact that
$\|\nabla D^2\phi^h\|_{L^2}\leq C\|D^3\phi\|_{L^2}$, we have
\begin{align*}
	J_4
	&\leq C\|\nabla D^2\phi^h\|_{L^2}
	\big(
	\|\nabla u\|_{L^3}\|D^2u\|_{L^6}
	+\|D^2u\|_{L^6}\|\nabla u\|_{L^3}
	\big) \\
    &\leq C\| D^3\phi^h\|_{L^2}
	\|D u\|_{H^1}\|D^3u\|_{L^2}\\
	&\leq 
	C\|D^3\phi\|_{L^2}\|D^3u\|_{L^2}.
\end{align*}
For the $J_5$--term, we directly obtain
\begin{align*}
	J_5
	\leq \|\nabla D^2\phi^h\|_{L^2}\|u\|_{L^\infty}
	\|\nabla D^2u\|_{L^2}
	\leq C\|D^3\phi\|_{L^2}\|D^3u\|_{L^2}.
\end{align*}
Moreover, using \eqref{gaodi}, we have
\[
\|D^2(u\cdot\nabla u)^\ell\|_{L^2}
\leq C\|\nabla D(u\cdot\nabla u)\|_{L^2}.
\]
Then
\begin{align*}
	J_6
	&\leq C\|\nabla D^2\phi^h\|_{L^2}
	\|\nabla D(u\cdot\nabla u)\|_{L^2} \\
	&\leq C\|D^3\phi\|_{L^2}
	\big(
	\|\nabla D u\|_{L^6}\|\nabla u\|_{L^3}
	+\|\nabla u\|_{L^3}\|D^2u\|_{L^6}
	+\|u\|_{L^\infty}\|\nabla D^2u\|_{L^2}
	\big) \\
	&\leq C\|D^3\phi\|_{L^2}\|D^3u\|_{L^2}.
\end{align*}
Consequently,
\begin{align}
	K_2
	\leq
	C\|D^3\phi\|_{L^2}\|D^3u\|_{L^2}.
	\label{K2estimate}
\end{align}

We next estimate the $K_3$--term. We write
\begin{align}
	K_3
	&=-\frac{\gamma-1}{2}
	\int_{\mathbb R^3}D^2u^h\cdot\nabla D^2(\phi\,\mathrm{div}u)^h\,\,\mathrm{d}x \notag\\
	&=-\frac{\gamma-1}{2}
	\int_{\mathbb R^3}D^2u^h\cdot\nabla D^2(\phi\,\mathrm{div}u)\,\,\mathrm{d}x \notag\\
	&\quad
	+\frac{\gamma-1}{2}
	\int_{\mathbb R^3}D^2u^h\cdot\nabla D^2(\phi\,\mathrm{div}u)^\ell\,\mathrm{d}x \notag\\
	&=-\frac{\gamma-1}{2}
	\int_{\mathbb R^3}D^2u^h\cdot
	[\nabla D^2,\phi]\,\mathrm{div}u\,\mathrm{d}x \notag\\
	&\quad
	-\frac{\gamma-1}{2}
	\int_{\mathbb R^3}D^2u^h\cdot
	\phi\,\nabla D^2\mathrm{div}u\,\mathrm{d}x \notag\\
	&\quad
	+\frac{\gamma-1}{2}
	\int_{\mathbb R^3}D^2u^h\cdot\nabla D^2(\phi\,\mathrm{div}u)^\ell\,\mathrm{d}x \notag\\
	&:=J_7+J_8+J_9 .
	\label{K3decomp}
\end{align}
For the $J_7-$term, Lemma \ref{jiaohuan} gives
\begin{align*}
	J_7
	&\leq \tilde{C}\|D^2u^h\|_{L^2}
	\big(
	\|\nabla\phi\|_{L^\infty}\|D^2\mathrm{div}u\|_{L^2}
	+\|D^3\phi\|_{L^2}\|\mathrm{div}u\|_{L^\infty}
	\big) \\
	&\leq  \tilde{C}\big(\|\nabla\phi\|_{L^\infty}
	+\|\nabla u\|_{L^\infty}\big)
	\big(\|D^3\phi\|_{L^2}^2+\|D^3u\|_{L^2}^2\big).
\end{align*}
For the $J_8-$term, integration by parts yields
\begin{align*}
	J_8
	&=\frac{\gamma-1}{2}
	\int_{\mathbb R^3}\mathrm{div}\big(\phi D^2u^h\big)
	D^2\mathrm{div}u\,\mathrm{d}x \\
	&=\frac{\gamma-1}{2}
	\int_{\mathbb R^3}\phi\,\mathrm{div}D^2u^h\,
	D^2\mathrm{div}u\,\mathrm{d}x
	+\frac{\gamma-1}{2}
	\int_{\mathbb R^3}\nabla\phi\cdot D^2u^h\,
	D^2\mathrm{div}u\,\mathrm{d}x,
\end{align*}
so that
\begin{align*}
	J_8
	&\leq  \tilde{C}\|\phi\|_{L^\infty}\|\nabla D^2u^h\|_{L^2}
	\|D^2\mathrm{div}u\|_{L^2}
	+\tilde{C}\|\nabla\phi\|_{L^\infty}\|D^2u^h\|_{L^2}
	\|D^2\mathrm{div}u\|_{L^2} \\
	&\leq  \tilde{C}\big(\|\phi\|_{L^\infty}+\|\nabla\phi\|_{L^\infty}\big)
	\|D^3u\|_{L^2}^2 .
\end{align*}
Finally, by \eqref{gaodi}, we have
\[
\|\nabla D^2(\phi\,\mathrm{div}u)^\ell\|_{L^2}
\leq C\|\nabla D(\phi\,\mathrm{div}u)\|_{L^2}.
\]
Therefore, we can estimate $J_9-$term as
\begin{align*}
	J_9
	&\leq  \tilde{C}\|D^2u^h\|_{L^2}
	\|\nabla D(\phi\,\mathrm{div}u)\|_{L^2} \\
	&\leq  \tilde{C}\|D^3u\|_{L^2}
	\big(
	\|\nabla D\phi\|_{L^6}\|\mathrm{div}u\|_{L^3}
	+\|\nabla\phi\|_{L^3}\|D\,\mathrm{div}u\|_{L^6}
	+\|\phi\|_{L^\infty}\|\nabla D\,\mathrm{div}u\|_{L^2}
	\big) \\
	&\leq  \tilde{C}\|D^3u\|_{L^2}
	\big(
	\|D^3\phi\|_{L^2}\|\nabla u\|_{H^1}
	+\|\nabla\phi\|_{H^1}\|D^3u\|_{L^2}
	+\|\phi\|_{L^\infty}\|D^3u\|_{L^2}
	\big) \\
	&\leq \tilde{C}\|D^3\phi\|_{L^2}\|D^3u\|_{L^2}
	+ \tilde{C}\big(\|\nabla\phi\|_{H^1}+\|\phi\|_{L^\infty}\big)
	\|D^3u\|_{L^2}^2 .
\end{align*}
Combining the estimates of $J_7$, $J_8$, and $J_9$, we obtain
\begin{align}
	K_3
	&\leq
	C\big(\|\phi\|_{L^\infty}
	+\|\nabla\phi\|_{L^\infty}
	+\|\nabla u\|_{L^\infty}
	+\|\nabla\phi\|_{H^1}\big)
	\big(\|D^3\phi\|_{L^2}^2+\|D^3u\|_{L^2}^2\big) \notag\\
	&\quad\,+C\|D^3\phi\|_{L^2}\|D^3u\|_{L^2}.
	\label{K3estimate}
\end{align}

For the $K_4-$term on the right-hand side of \eqref{d3ah}, we have
		\begin{equation}\label{estimateK4}
			\begin{aligned}
				K_4&=-\frac{\gamma-1}{2}\int_{\mathbb{R}^3} \nabla D^2 \phi^h\cdot  D^2 (\phi \nabla \phi)^h\, \mathrm{d}x \\
				&\leq \frac{C(\gamma-1)}{2} \big( \|\phi\|_{L^{\infty}}\|D^3 \phi\|_{L^2}^2+\|\nabla \phi\|_{L^{\infty}}\|D^3 \phi\|_{L^2}^2 \big)\\
                &\leq \tilde{C} \big( \|\phi\|_{L^{\infty}}\|D^3 \phi\|_{L^2}^2+\|\nabla \phi\|_{L^{\infty}}\|D^3 \phi\|_{L^2}^2 \big).
			\end{aligned}
		\end{equation}

		Putting \eqref{4xiang}, \eqref{K2estimate}, and \eqref{K3estimate}--\eqref{estimateK4} 
        into \eqref{d3ah} and noticing
		\begin{align*}
		\|(\nabla \phi,\nabla u)\|_{L^{\infty}}
		\leq C\|(D^2 \phi,D^2 u)\|^{\frac{1}{2}}_{L^2}
		\|(D^3 \phi,D^3 u)\|^{\frac{1}{2}}_{L^2}
		\leq \tilde{C} \delta_0^{\frac{9}{11}},
		\end{align*}
		we obtain
		\begin{align*}
			&\|D^3 \phi^h\|_{L^2}^2
			+\frac{\mathrm{d}}{\mathrm{d}t}\int_{\mathbb{R}^3} D^2 u^h\cdot \nabla D^2 \phi^h\, \mathrm{d}x \\
			&\leq \epsilon\|D^3 \phi\|_{L^2}^2
			+C_{\epsilon}\|D^3 u\|_{L^2}^2
			+\tilde{C}\big(\|\nabla \phi\|_{L^{\infty}}
			+\|\nabla u\|_{L^{\infty}}\big)
			\|(D^3 \phi,D^3 u)\|_{L^2}^2\\
			&\leq (\epsilon+\tilde{C}\delta_0^{\frac{9}{11}})\big(\|D^3 \phi^{h}\|_{L^2}^2
			+\|D^3 \phi^{\ell}\|_{L^2}^2\big)
			+(C_{\epsilon}+\tilde{C}\delta_0^{\frac{9}{11}})\|D^3 u\|_{L^2}^2.
		\end{align*}
		Choosing $\epsilon$ sufficiently small such that $\epsilon+\tilde{C}\delta_0^{\frac{9}{11}}\leq\frac{1}{2}$, then there exists a positive constant $\tilde{C}_6$ such that
		\begin{align}\label{uhah}
			\frac{1}{2}\|D^3 \phi^h\|_{L^2}^2+
			\frac{\mathrm{d}}{\mathrm{d}t}
			\int_{\mathbb{R}^3} D^2 u^h\cdot \nabla D^2 \phi^h\, \mathrm{d}x
			\leq \frac{1}{2}\|D^3 \phi^{\ell}\|_{L^2}^2+\tilde{C}_6\|D^3 u\|_{L^2}^2.
		\end{align}
		
		On the other hand, we deduce from \eqref{aualpha1} that
		\begin{align*}
			\frac{1}{2}\frac{\mathrm{d}}{\mathrm{d}t}
			\|(D^3 \phi,D^3 u)\|_{L^2}^2
			+ \|D^3 u\|_{L^2}^2
			&{}\leq \tilde{C} \big(\|\nabla \phi\|_{L^{\infty}}
			+\|\nabla u\|_{L^{\infty}}\big)\|(D^3 \phi,D^3 u)\|_{L^2}^2\\
			&{}\leq \tilde{C}\|(\nabla D  \phi,\nabla D u)\|^{\frac{1}{2}}_{L^2}
			\|(\nabla D^2 \phi,\nabla D^2 u)\|^{\frac{1}{2}}_{L^2}
			\|(D^3 \phi,D^3 u)\|_{L^2}^2\\
			&{}\leq \tilde{C}\delta_0^{\frac{9}{11}}\|(D^3 \phi,D^3 u)\|_{L^2}^2,
		\end{align*}
		and further obtain from the smallness of $\delta_0$ that
		\begin{align}\label{a3u3}
			\frac{\mathrm{d}}{\mathrm{d}t}
			\|(D^3 \phi, D^3 u)\|_{L^2}^2
			+\|D^3 u\|_{L^2}^2
			\leq \tilde{C}\delta_0^{\frac{9}{11}}\|D^3 \phi\|_{L^2}^2.
		\end{align}
		Multiplying \eqref{uhah} by $\beta_2>0$ and adding the resulting equations to \eqref{a3u3}, we deduce
		\begin{align*}
			&\frac{\mathrm{d}}{\mathrm{d}t}\Big(
			\|(D^3 \phi, D^3 u)\|_{L^2}^2
			+\beta_2 \int_{\mathbb{R}^3} D^2 u^h\cdot \nabla D^2 \phi^h\, \mathrm{d}x \Big)
			+ \frac{1}{2}\beta_2\|D^3 \phi^h\|_{L^2}^2
			+\|D^3 u\|_{L^2}^2\\
			&\leq \big(\frac{1}{2}\beta_2+\tilde{C}\delta_0^{\frac{9}{11}}\big)\|D^3 \phi^{\ell}\|_{L^2}^2
			+ \tilde{C}_6\beta_2\|D^3 u\|_{L^2}^2+\tilde{C}\delta_0^{\frac{9}{11}}\|D^3 \phi^{h}\|_{L^2}^2.
		\end{align*}
		Choosing $\beta_2$ suitably small such that
		\begin{align*}
		\tilde{C}_6\beta_2\leq \frac{1}{2} ,\quad \quad \tilde{C}\delta_0^{\frac{9}{11}}\leq \frac{1}{4}\beta_2.
		\end{align*}
Then		
		\begin{align}\label{G3U}
		\frac{\mathrm{d}}{\mathrm{d}t}\Big(
		\|(D^3 \phi, D^3 u)\|_{L^2}^2
		+\beta_2 \int_{\mathbb{R}^3} D^2 u^h\cdot \nabla D^2 \phi^h\, \mathrm{d}x \Big)
		+C\delta_0^{\frac{9}{11}}\|(D^3\phi, D^3u)\|_{L^2}^2\leq \tilde{C}\delta_0^{\frac{9}{11}}\|D^3\phi^\ell\|_{L^2}^2.
	\end{align}			
Define 
\begin{align}
A_2(t)=	\|(D^3 \phi, D^3 u)\|_{L^2}^2
+\beta_2 \int_{\mathbb{R}^3} D^2 u^h\cdot \nabla D^2 \phi^h\, \mathrm{d}x.	
\end{align}		
Due to the smallness of $\beta_2$ and the Poincar\'{e} inequality for the high-frequency part, we have 
	\begin{align}\label{G4U}
A_2(t)=\|(D^3 \phi, D^3 u)\|_{L^2}^2
	+\beta_2 \int_{\mathbb{R}^3} D^2 u^h\cdot \nabla D^2 \phi^h\, \mathrm{d}x
	\sim \|(D^3 \phi, D^3 u)\|_{L^2}^2.
\end{align}
Solving the differential inequality \eqref{G3U} yields
\begin{align}
\sqrt{A_2(t)}\leq Ce^{-C\delta_0^{\frac{9}{11}}t}\|(D^3\phi_0, D^3u_0)\|_{L^2}+\tilde{C}\delta_0^{\frac{9}{11}}\int_0^t e^{-C\delta_0^{\frac{9}{11}}(t-\tau)}\|D^3\phi^\ell(\tau)\|_{L^2}\mathrm{d}\tau.	
\end{align}	
In terms of \eqref{G4U}, we finally obtain
\begin{align}
 \|(D^3 \phi, D^3 u)\|_{L^2}^2\leq Ce^{-C\delta_0^{\frac{9}{11}}t}\|(D^3\phi_0, D^3u_0)\|_{L^2}^2+\tilde{C}\delta_0^{\frac{9}{11}}\int_0^t e^{-C\delta_0^{\frac{9}{11}}(t-\tau)}\|D^3\phi^\ell(\tau)\|_{L^2}^2\mathrm{d}\tau.	
\end{align}	
Thus, we conclude \eqref{dadu-3}.
	\end{proof}
    
	Collecting the results in Lemmas \ref{d12-lem1}--\ref{d3-lem1} leads to
	\begin{align}\label{d123}
	 &\|(D^k \phi, D^k u)\|_{H^{3-k}}^2\nonumber\\
     &\leq Ce^{-C\delta_0^{\frac{9}{11}}t}\|(D^k\phi_0, D^ku_0)\|_{H^{3-k}}^2
     +\tilde{C}\delta_0^{\frac{9}{11}}\int_0^t e^{-C\delta_0^{\frac{9}{11}}(t-\tau)}\|D^k\phi^\ell(\tau)\|_{L^2}^2\mathrm{d}\tau\quad\,\, 
     \text{for \;$k=1,2,3$.}
	\end{align}
	
	Now, we are in a position to obtain the decay rates of $(D^k \phi, D^k u)$.
	
	\begin{lem}\label{decay-lem-1}
		Under assumption \eqref{priori}, 
		\begin{align*}
			&\|D^k \phi(t)\|_{L^2}\leq \tilde{C} N_0(1+t)^{-\frac{k}{2}} \qquad\,\,\,\,\,\text{for \;$k=1,2,3,$}\\
			&\|D^{k}u(t)\|_{L^2}\leq \tilde{C} N_0(1+t)^{-\frac{1+k}{2}} \qquad\text{for \;$k=0,1,2,$}\\
			&\|D^{3}u(t)\|_{L^2}\leq \tilde{C} N_0(1+t)^{-\frac{3}{2}},
		\end{align*}
		where $N_0:= \|(\phi_0,u_0)\|_{H^3}$.
	\end{lem}
	\begin{proof}
		Based on Lemma \ref{gl-lam}, 
        we define the time-weighted functional $Y(t)$ as
		\begin{align*}
			Y(t)=&\sup\limits_{0\leq \tau\leq t}\Big\{
			\sum_{k=1}^{3}(1+\tau)^{\frac{k}{2}}\|D^k \phi(\tau)\|_{L^2}
			+\sum_{k=0}^{2}(1+\tau)^{\frac{k+1}{2}}\|D^k u(\tau)\|_{L^2}
			+(1+\tau)^{\frac{3}{2}}\|D^3 u(\tau)\|_{L^2}
			\Big\}.
		\end{align*}
		It follows from the definition of  $Y(t)$ that
		\begin{align*}
			\|D^k \phi(t)\|_{L^2}&\leq (1+t)^{-\frac{k}{2}}Y(t) \qquad\quad\text{for \;$k=1,2,3,$}\\
			\|D^{k}u(t)\|_{L^2}&\leq (1+t)^{-\frac{1+k}{2}}Y(t)  \qquad\,\text{for \;$k=0,1,2,$}\\
			\|D^{3}u(t)\|_{L^2}&\leq (1+t)^{-\frac{3}{2}}Y(t).
		\end{align*}
It remains to prove the boundedness of $Y(t)$.
        
		Since
		\begin{align*}
			\|f_1(t)\|_{L^2}
			&
			\leq \|u\|_{L^2}\|\nabla \phi\|_{L^{\infty}}+\frac{\gamma-1}{2}\|\phi\|_{L^\infty}\|\nabla u\|_{L^2}\\
			&\leq C\|u\|_{L^2}\|D^2 \phi\|_{L^2}^{\frac{1}{2}}
			\|D^3 \phi\|_{L^2}^{\frac{1}{2}}+C(\gamma-1)\|D \phi\|_{L^2}^{\frac{1}{2}}
			\|D^2\phi \|_{L^2}^{\frac{1}{2}}\|u\|_{L^2}^{\frac{2}{3}}
			\|D^3 u\|_{L^2}^{\frac{1}{3}} \\
			&\leq C\|u\|_{L^2}\|D^2 \phi\|_{L^2}^{\frac{1}{2}}
			\|D^3 \phi\|_{L^2}^{\frac{1}{3}}\|D^3 \phi\|_{L^2}^{\frac{1}{6}}+C(\gamma-1)\|D\phi\|_{L^2}^{\frac{1}{2}}
			\|D^2\phi \|_{L^2}^{\frac{1}{2}}\|u\|_{L^2}^{\frac{2}{3}}
			\|D^3 u\|_{L^2}^{\frac{1}{6}}\|D^3 u\|_{L^2}^{\frac{1}{6}}\\
			&\leq \tilde{C}(1+t)^{-\frac{3}{2}}\delta_0^{\frac{1}{6}}Y^{\frac{11}{6}}(t)+\tilde{C}(1+t)^{-\frac{4}{3}}\delta_0^{\frac{1}{6}}Y^{\frac{11}{6}}(t)\\
			&\leq \tilde{C} (1+t)^{-\frac{4}{3}}\delta_0^{\frac{1}{6}}Y^{\frac{11}{6}}(t),
		\end{align*}
		\begin{align*}
			\|f_2(t)\|_{L^2}
			&
			\leq \|u\|_{L^2}\|\nabla u\|_{L^{\infty}}+\frac{\gamma-1}{2}\|\phi\|_{L^\infty}\|\nabla \phi\|_{L^2}\\
			&\leq C\|u\|_{L^2}\|D^2 u\|_{L^2}^{\frac{1}{2}}
			\|D^3 u\|_{L^2}^{\frac{1}{2}}+C(\gamma-1)\|D\phi\|_{L^2}^{\frac{1}{2}}
			\|D^2 \phi\|_{L^2}^{\frac{1}{2}}\|D\phi\|_{L^2} \\
			&\leq C\|u\|_{L^2}\|D^2 u\|_{L^2}^{\frac{1}{2}}
			\|D^3 u\|_{L^2}^{\frac{1}{3}}\|D^3 u\|_{L^2}^{\frac{1}{6}}+C(\gamma-1)\|D \phi\|_{L^2}^{\frac{7}{4}}\|D^3 \phi\|_{L^2}^{\frac{1}{8}}\|D^3 \phi\|_{L^2}^{\frac{1}{8}}\\
		&\leq \tilde{C}(1+t)^{-\frac{7}{4}}\delta_0^{\frac{1}{6}}Y^{\frac{11}{6}}(t)+\tilde{C}(1+t)^{-\frac{17}{16}}\delta_0^{\frac{1}{8}}Y^{\frac{15}{8}}(t),
		\end{align*}
		we obtain from the expression of solution \eqref{duhamel} that
		\begin{align}
			\|D \phi^{\ell}\|_{L^2}
			&\leq \|D G_{11}^{\ell}\ast \phi_0\|_{L^2}
			+\|D G_{12}^{\ell}\ast u_0\|_{L^2}
			+\int^t_0 \|D G_{11}^{\ell}(t-\tau)\ast f_1(\tau)\|_{L^2}\mathrm{d}\tau \notag\\
			&\quad + \int^t_0 \|D G_{12}^{\ell}(t-\tau)\ast f_2(\tau)\|_{L^2}\mathrm{d}\tau \notag\\
			&\leq C(1+t)^{-\frac{1}{2}}\|\phi_0\|_{L^2}
			+C(1+t)^{-1}\|u_0\|_{L^2}
			+C \int^t_0 (1+t-\tau)^{-\frac{1}{2}}\|f_1(\tau)\|_{L^2}\mathrm{d}\tau \notag\\
			&\quad + C \int^t_0 (1+t-\tau)^{-1}\|f_2(\tau)\|_{L^2}\mathrm{d}\tau \notag\\
			&\leq C(1+t)^{-\frac{1}{2}}\big(\|\phi_0\|_{L^2}
			+\|u_0\|_{L^2}\big)\nonumber\\
			&\quad + \tilde{C} \int^t_0 (1+t-\tau)^{-\frac{1}{2}} (1+\tau)^{-\frac{4}{3}}\delta_0^{\frac{1}{6}}Y^{\frac{11}{6}}(t)
			\mathrm{d}\tau \notag\\
			&\quad + \tilde{C} \int^t_0 (1+t-\tau)^{-1}\big(
			(1+\tau)^{-\frac{7}{4}}\delta_0^{\frac{1}{6}}Y^{\frac{11}{6}}(t)+(1+\tau)^{-\frac{17}{16}}\delta_0^{\frac{1}{8}}Y^{\frac{15}{8}}(t)
			\big)\mathrm{d}\tau \notag\\
			&\leq \tilde{C}(1+t)^{-\frac{1}{2}}\big(\|\phi_0\|_{L^2}
			+\|u_0\|_{L^2}
			+\delta_0^{\frac{1}{6}}Y^{\frac{11}{6}}(t)+\delta_0^{\frac{1}{8}}Y^{\frac{15}{8}}(t)\big), \label{al-1}
		\end{align}
		where we have used Lemma \ref{gl-lam}. After a similar method, we achieve that
		\begin{align}
			&\|D^2 \phi^{\ell}(t)\|_{L^2}
			\leq \tilde{C}(1+t)^{-1}\big(\|\phi_0\|_{L^2}
			+\|u_0\|_{L^2}
			+\delta_0^{\frac{1}{6}}Y^{\frac{11}{6}}(t)+\delta_0^{\frac{1}{8}}Y^{\frac{15}{8}}(t)\big),\label{al-2}\\
		&\|D^3 \phi^{\ell}(t)\|_{L^2}
			\leq \tilde{C}(1+t)^{-\frac{3}{2}}\big(\|\phi_0\|_{L^2}
			+\|u_0\|_{L^2}
		+\delta_0^{\frac{1}{6}}Y^{\frac{11}{6}}(t)+\delta_0^{\frac{1}{8}}Y^{\frac{15}{8}}(t)\big),\label{al-3}
		\end{align}
	where we used the integration by parts for the third-derivatives to get \eqref{al-3}.
		
		Substituting \eqref{al-1} into \eqref{dadu-1} and then using Gr\"{o}nwall's inequality yields
		\begin{align}\label{022701}
			\|(D \phi,D u)(t)\|_{H^2}
			\leq \tilde{C}(1+t)^{-\frac{1}{2}}\Big(N_0
			+\delta_0^{\frac{1}{6}}Y^{\frac{11}{6}}(t)+\delta_0^{\frac{1}{8}}Y^{\frac{15}{8}}(t)\Big).
		\end{align}
		Similarly, substituting \eqref{al-2}--\eqref{al-3} into \eqref{dadu-2},\eqref{dadu-3} respectively, we arrive at
		\begin{align}
			&\|(D^2 \phi,D^2 u)(t)\|_{H^1}
			\leq \tilde{C}(1+t)^{-1}\big(N_0
			+\delta_0^{\frac{1}{6}}Y^{\frac{11}{6}}(t)
             +\delta_0^{\frac{1}{8}}Y^{\frac{15}{8}}(t)\big),\label{022702}\\
		&\|(D^3 \phi, D^3 u)(t)\|_{L^2}
			\leq \tilde{C}(1+t)^{-\frac{3}{2}} \big(N_0
			+\delta_0^{\frac{1}{6}}Y^{\frac{11}{6}}(t)+\delta_0^{\frac{1}{8}}Y^{\frac{15}{8}}(t)\big).
		\end{align}
		By the damping structure of velocity, we solve \eqref{au1} as
		\begin{align}\label{022602}
			u=e^{-t}u_0-\int_0^t e^{-(t-\tau)}\big(u\cdot\nabla u+\frac{\gamma-1}{2}\phi\cdot \nabla \phi+\sqrt{\gamma}\nabla \phi\big)(\tau) \mathrm{d}\tau.
		\end{align}
		Then we deduce from  \eqref{022701}--\eqref{022702} and \eqref{022602} that
		\begin{align*}
			\|D^{k}u(t)\|_{L^2}\leq \tilde{C}(1+t)^{-\frac{1+k}{2}}\Big(N_0
			+\delta_0^{\frac{1}{6}}Y^{\frac{11}{6}}(t)+\delta_0^{\frac{1}{8}}Y^{\frac{15}{8}}(t)\Big)
            \quad \quad\text{for \;$k=0,1,2.$}
		\end{align*}	
		According to the definition of $Y(t)$, the combination of the above results yields
		\begin{align}
			Y(t)
			\leq CN_0
			+\tilde{C}\delta_0^{\frac{1}{6}}Y^{\frac{11}{6}}(t)+\tilde{C}\delta_0^{\frac{1}{8}}Y^{\frac{15}{8}}(t),\label{Y(t)}
		\end{align}
		where $N_0 := \|(\phi_0,u_0)\|_{H^3}$.
		Due to the smallness of  $\delta_0$, we deduce from \eqref{Y(t)} that
		\begin{align*}
			Y(t)\leq \tilde{C}N_0.
		\end{align*}
		Therefore, we obtain
		\begin{align*}
			&\|D^k \phi(t)\|_{L^2}\leq \tilde{C} N_0(1+t)^{-\frac{k}{2}} \qquad\quad\text{for \;$k=1,2,3,$}\\
			&\|D^{k}u(t)\|_{L^2}\leq \tilde{C} N_0(1+t)^{-\frac{1+k}{2}} \qquad\,\,\text{for \;$k=0,1,2,$}\\
		&\|D^{3}u(t)\|_{L^2}\leq \tilde{C} N_0(1+t)^{-\frac{3}{2}}.
		\end{align*}
		This completes the proof.
	\end{proof}
	It should be pointed out that the above decay rate for the highest-order derivative of $u$ is not optimal. In fact, it can be improved. In the following lemma, we obtain the optimal decay rates of $\|D^3u\|_{L^2}$.
	
	\begin{lem}\label{decay-lem-2}
		Under the {\it a priori} assumption \eqref{priori},
		$$
		\|D^{3}u(t)\|_{L^2}\ \le\ C \,(N_0+N_0^2)\,(1+t)^{-2}.
		$$ 
	\end{lem}
	
	\begin{proof}
		Applying $D^3$ to \eqref{high-f}, we have
		\begin{equation}\label{D3auf}
			\left\{
			\begin{aligned}
				& \partial_t D^3 \phi^h + \sqrt{\gamma}{\mathrm{div}}D^3u^h= D^3f^h_1,\\
				& \partial_t D^3u^h+\sqrt{\gamma}\nabla D^3\phi^h+D^3u^h =D^3f^h_2.
			\end{aligned}
			\right.
		\end{equation}
		Taking the $L^2$-inner product of \eqref{D3auf} with $(D^3\phi^{h}, D^3 u^{h})$ yields
			\begin{align}
				&\frac{1}{2}\frac{\mathrm{d}}{\mathrm{d}t}\|(D^3 \phi^{h},D^3u^{h})\|^2_{L^2}
				+\|D^3u^{h}\|^2_{L^2} \nonumber\\
				&= \int_{\mathbb{R}^3}D^3\phi^{h}D^3 f_1^{h} \, \mathrm{d}x
				+\int_{\mathbb{R}^3}D^3u^{h}\cdot D^3f_2^{h} \, \mathrm{d}x \nonumber\\
				&=-\int_{\mathbb{R}^3}D^3\phi^{h}D^3 (u\cdot \nabla \phi)^{h} \, \mathrm{d}x -\frac{\gamma-1}{2}  \int_{\mathbb{R}^3}D^3\phi^{h}D^3 (\phi  {\mathrm{div}} u)^{h} \, \mathrm{d}x \nonumber\\
				&\quad\quad-\int_{\mathbb{R}^3}D^3u^{h}\cdot D^3(u\cdot \nabla u)^{h} \, \mathrm{d}x -\frac{\gamma-1}{2}  \int_{\mathbb{R}^3}D^3u^{h}\cdot D^3(\phi\cdot \nabla \phi)^{h} \, \mathrm{d}x \nonumber\\
				&= -\int_{\mathbb{R}^3}D^3\phi^{h}D^3 (u\cdot \nabla \phi)^{h} \, \mathrm{d}x -\int_{\mathbb{R}^3}D^3u^{h}\cdot D^3(u\cdot \nabla u)^{h} \, \mathrm{d}x \nonumber\\
				& \quad\quad -\frac{\gamma-1}{2} \int_{\mathbb{R}^3}\Big( D^3\phi^{h}D^3 (\phi  {\mathrm{div}} u)^{h} +D^3 u^{h}\cdot D^3(\phi\nabla \phi)^{h} \Big)\, \mathrm{d}x \nonumber\\
				&:= I_1+I_2+I_3.\label{D3hau}
			\end{align}
		For the first term on the right-hand side of \eqref{D3hau}, we use Lemmas \ref{jiaohuan}, \ref{kongzhi},
        and \ref{decay-lem-1} to derive
		\begin{align*}
        I_1&=-\int_{\mathbb{R}^3}D^3\phi^{h}D^3 (u\cdot \nabla \phi)^{h} \, \mathrm{d}x \\
			&=\int_{\mathbb{R}^3}D^3\phi^{h}D^3 (u\cdot \nabla \phi) \, \mathrm{d}x
			-\int_{\mathbb{R}^3}D^3\phi^{h}D^3 (u\cdot \nabla \phi)^{\ell} \, \mathrm{d}x\\
			&= \int_{\mathbb{R}^3}D^3\phi^{h}\Big(D^3 (u\cdot \nabla \phi)
			- u\cdot \nabla D^3\phi\Big)\, \mathrm{d}x\\
			& \quad +\int_{\mathbb{R}^3}D^3\phi^{h}(u\cdot \nabla D^3\phi^{h})\, \mathrm{d}x
			+\int_{\mathbb{R}^3}D^3\phi^{h}(u\cdot \nabla D^3\phi^{\ell})\, \mathrm{d}x\\
			&\quad -\int_{\mathbb{R}^3}D^3\phi^{h}D^3 (u\cdot \nabla \phi)^{\ell} \, \mathrm{d}x \\
			&\leq \|D^3 \phi^h\|_{L^2}\|[D^3,u\cdot \nabla]\phi\|_{L^2}
			+C\|\nabla u\|_{L^{\infty}}\|D^3 \phi^{h}\|^2_{L^2}\\
			&\quad +\|u\|_{L^{\infty}}\|\nabla D^3 \phi^{\ell}\|_{L^2}\|D^3 \phi^h \|_{L^2}
			+\|D^3 \phi^h\|_{L^2}\|D^3(u\cdot \nabla \phi)^{\ell}\|_{L^2} \\
			&\leq C\|D^3 \phi^h\|_{L^2}\big(\|\nabla u\|_{L^{\infty}}\|D^3 \phi\|_{L^2}
			+\|\nabla \phi\|_{L^{\infty}}\|D^3 u\|_{L^2}\big)\\
			& \quad+C\|\nabla u\|_{L^{\infty}}\|D^3 \phi^{h}\|^2_{L^2}
			+\|u\|_{L^{\infty}}\|\nabla D^3 \phi^{\ell}\|_{L^2}\|D^3 \phi^h \|_{L^2}\\
			&\quad +C\|D^3 \phi^h\|_{L^2}\big(\|u\|_{\infty}\|D^3 \phi\|_{L^2}
			+\|\nabla \phi\|_{L^\infty}\|D^2 u\|_{L^2}+\|\nabla u\|_{L^\infty}\|D^2\phi\|_{L^2}\big)\\
			&\leq C\|D^3 \phi\|_{L^2}\big(\|\nabla u\|_{L^{\infty}}\|D^3 \phi\|_{L^2}
			+\|\nabla \phi\|_{L^{\infty}}\|D^3 u\|_{L^2}\big)\\
			& \quad+C\|\nabla u\|_{L^{\infty}}\|D^3 \phi\|^2_{L^2}
			+C\|u\|_{L^{\infty}}\|D^3 \phi\|_{L^2}^2\\
			&\quad+C\|D^3 \phi\|_{L^2}\big(\|u\|_{L^\infty}\|D^3 \phi\|_{L^2}+\|\nabla u\|_{L^\infty}\|D^2\phi\|_{L^2}
			+\|\nabla \phi\|_{L^\infty}\|D^2 u\|_{L^2}\big)\\
			&\leq\tilde{C}N_0^3(1+t)^{-\frac{17}{4}}.
		\end{align*}
		
		Similarly, for the second term on the right-hand side of \eqref{D3hau}, we have
		\begin{align*}
			I_2&=-\int_{\mathbb{R}^3}D^3u^{h}\cdot D^3(u\cdot \nabla u)^{h} \, \mathrm{d}x\\
			&\leq C\|D^3 u^h\|_{L^2}\|\nabla u\|_{L^{\infty}}\|D^3 u\|_{L^2}
			\\
			& \quad +C\|\nabla u\|_{L^{\infty}}\|D^3 u^{h}\|^2_{L^2}
			+\|u\|_{L^{\infty}}\|\nabla D^3 u^{\ell}\|_{L^2}\|D^3 u^h \|_{L^2}\\
			&\quad+C\|D^3 u^h\|_{L^2}\big(\|u\|_{\infty}\|D^3 u\|_{L^2}
			+\|\nabla u\|_{L^\infty}\|D^2 u\|_{L^2}\big)\\
			&\leq \tilde{C}N_0^3(1+t)^{-\frac{17}{4}}.
		\end{align*}
		
		For the third term on the right-hand side of \eqref{D3hau}, we deduce
		\begin{align*}
			I_3&=-\frac{\gamma-1}{2} \int_{\mathbb{R}^3}\big( D^3\phi^{h}D^3 (\phi \,{\mathrm{div}} u)^{h} +D^3 u^{h}\cdot D^3(\phi\nabla \phi)^{h} \big) \, \mathrm{d}x \\
			&=-\frac{\gamma-1}{2} \int_{\mathbb{R}^3} \big( D^3\phi^{h} \cdot \phi D^3  {\mathrm{div}} u^{h} +D^3 u^{h}\cdot \phi D^3 \nabla \phi^{h} \big)\, \mathrm{d}x \\
			& \quad-\frac{\gamma-1}{2} \int_{\mathbb{R}^3} \big( D^3\phi^{h} [D^3,\phi {\mathrm{div}}]u +D^3 u^{h}\cdot [D^3,\phi \nabla ] \phi\big)\, \mathrm{d}x \\
            &\quad-\frac{\gamma-1}{2} \int_{\mathbb{R}^3} \big(D^3\phi^h\cdot\phi D^3\text{div}u^\ell+D^3u^h\phi D^3\nabla\phi^\ell\big) \mathrm{d}x\\
            &\quad+\frac{\gamma-1}{2} \int_{\mathbb{R}^3}\big( D^3\phi^{h}D^3 (\phi \,{\mathrm{div}} u)^{\ell} +D^3 u^{h}\cdot D^3(\phi\nabla \phi)^{\ell} \big) \, \mathrm{d}x \\
			&=-\frac{\gamma-1}{2} \int_{\mathbb{R}^3}\phi {\mathrm{div}}(D^3\phi^{h} D^3 u^{h} )\, \mathrm{d}x
            -\frac{\gamma-1}{2} \int_{\mathbb{R}^3} \big( D^3\phi^{h} [D^3,\phi {\mathrm{div}}]u +D^3 u^{h}\cdot [D^3,\phi \nabla ] \phi\big)\, \mathrm{d}x \\
            &\quad-\frac{\gamma-1}{2} \int_{\mathbb{R}^3} \big(D^3\phi^h\cdot\phi D^3\text{div}u^\ell+D^3u^h\phi D^3\nabla\phi^\ell\big) \mathrm{d}x\\
             &\quad+\frac{\gamma-1}{2} \int_{\mathbb{R}^3}\big( D^3\phi^{h}D^3 (\phi \,{\mathrm{div}} u)^{\ell} +D^3 u^{h}\cdot D^3(\phi\nabla \phi)^{\ell} \big) \, \mathrm{d}x \\
			&\leq\tilde{C} \big( \|\nabla \phi\|_{L^{\infty}}\|D^3 \phi^h\|_{L^2}\|D^3 u^h\|_{L^2}
+\|\nabla\phi\|_{L^\infty}\|D^3u\|_{L^2}\|D^3\phi^h\|_{L^2}\\
&\quad+\|\text{div} u\|_{L^{\infty}}\|D^3 \phi^h\|_{L^2}\|D^3u\|_{L^2}+\|\nabla\phi\|_{L^\infty}\|D^3\phi\|_{L^2}\|D^3u^h\|_{L^2}
 +\|\phi\|_{L^\infty}\|D^3\phi\|_{L^2}\|D^3u\|_{L^2}\big)\\
&\quad +\tilde{C}\big(\|D^3\phi^h\|_{L^2}\|D^2(\phi\text{div}u)\|_{L^2}+\|D^3u^h\|_{L^2}
\|D^2(\phi\nabla\phi)\|_{L^2}\big)\\
			&\leq\tilde{C} N_0^3(1+t)^{-\frac{17}{4}}.
		\end{align*}
		Then we arrive at
		\begin{align}\label{D3hau-t}
			\frac{1}{2}\frac{\mathrm{d}}{\mathrm{d}t}\|(D^3 \phi^{h},D^3u^{h})\|^2_{L^2}
			+\|D^3u^{h}\|^2_{L^2}
			\leq \tilde{C} N_0^3(1+t)^{-\frac{17}{4}}.
		\end{align}
		
		Now, we come back to \eqref{d3ah} to obtain the more accurate estimates.
		For the last two terms on the right-hand side of \eqref{d3ah}, we follow the method used above to obtain 
		\begin{align}\label{lasttwo}
			\Big|\int_{\mathbb{R}^3} D^2 u^h\cdot  \nabla D^2 f^h_1
			\, \mathrm{d}x\Big|
			+\Big|\int_{\mathbb{R}^3} \nabla D^2 \phi^h\cdot D^2 f^h_2
			\, \mathrm{d}x\Big|
			\leq \tilde{C} N_0^3(1+t)^{-\frac{17}{4}}.
		\end{align}
		Putting \eqref{lasttwo} into \eqref{d3ah} and choosing $\epsilon=\frac{1}{2}$, we obtain
		\begin{equation}\label{d3au-2}
			\frac{1}{2}\|D^3 \phi^h\|_{L^2}^2
			+\frac{\mathrm{d}}{\mathrm{d}t}\int_{\mathbb{R}^3} D^2 u^h\cdot \nabla D^2 \phi^h\, \mathrm{d}x \\
			\leq \tilde{C}\|D^3 u^h\|^2_{L^2}+\tilde{C} N_0^3(1+t)^{-\frac{17}{4}}.
		\end{equation}
		Multiplying \eqref{d3au-2} by $\beta_3>0$ and adding the resulting equations to \eqref{D3hau-t} yield
		\begin{align*}
			&\frac{\mathrm{d}}{\mathrm{d}t}
			\Big(\frac{1}{2}\|(D^3 \phi^{h},D^3u^{h})\|^2_{L^2}
			+\beta_3\int_{\mathbb{R}^3} D^2 u^h\cdot \nabla D^2 \phi^h\, \mathrm{d}x\Big)\\
			&\quad+(1-\tilde{C}\beta_3)\|D^3u^{h}\|^2_{L^2}+\frac{\beta_3}{2}\|D^3 \phi^h\|_{L^2}^2 \\
			&\leq \tilde{C}(1+\beta_3)N_0^3(1+t)^{-\frac{17}{4}}.
		\end{align*}
	Let
	\begin{align*}
	A_2(t)=	\frac{1}{2}\|(D^3 \phi^{h},D^3u^{h})\|^2_{L^2}
	+\beta_3\int_{\mathbb{R}^3} D^2 u^h\cdot \nabla D^2 \phi^h\, \mathrm{d}x.
	\end{align*}		
	Choosing $\beta_3$ suitably small such that 
	\begin{align}\label{AD2}
A_2(t)\sim 	\|(D^3 \phi^{h},D^3u^{h})\|^2_{L^2},
	\end{align}	
	 and  $1-\tilde{C}\beta_3 \geq\frac{1}{2}$, we obtain from the above inequality that
		\begin{align*}
			\frac{\mathrm{d}}{\mathrm{d}t}A_2(t)+CA_2(t)\leq \tilde{C}N_0^3(1+t)^{-\frac{17}{4}}.
		\end{align*}
		Then it follows from the Gr\"{o}nwall inequality and \eqref{AD2} that
		\begin{align}\label{D3u-h}
			\|(D^3 \phi^{h},D^3u^{h})\|_{L^2}
			\leq \tilde{C}N_0^{\frac{3}{2}}(1+t)^{-\frac{17}{8}}\leq \tilde{C}(N_0+N_0^2)(1+t)^{-\frac{17}{8}}.
		\end{align}
		
		For the low-frequency part of $D^3 u$, we obtain from Lemmas \ref{gl-lam} and \ref{decay-lem-1} that
        \begin{align}
			\|D^3 u^{\ell}\|_{L^2}
			&\leq \|D^3 G_{21}^{\ell}\ast \phi_0\|_{L^2}
			+\|D^3 G_{22}^{\ell}\ast u_0\|_{L^2}
			+\int^t_0 \|D^3 G_{21}^{\ell}(t-\tau)\ast f_1(\tau)\|_{L^2}\mathrm{d}\tau \nonumber\\
			&\quad + \int^t_0 \|D^3 G_{22}^{\ell}(t-\tau)\ast f_2(\tau)\|_{L^2}\mathrm{d}\tau \nonumber\\
			&\leq C(1+t)^{-2}\|\phi_0\|_{L^2}
			+C(1+t)^{-\frac{5}{2}}\|u_0\|_{L^2}
			+\int^{\frac{t}{2}}_0\|D^3G_{21}^\ell(t-\tau)\ast f_1(\tau)\|_{L^2}\mathrm{d}\tau \nonumber\\
			&\quad+\int_{\frac{t}{2}}^t\|D^2G_{21}^\ell(t-\tau)\ast D f_1(\tau)\|_{L^2}\mathrm{d}\tau
			+\int^{\frac{t}{2}}_0\|D^3G_{22}^\ell(t-\tau)\ast f_2(\tau)\|_{L^2}\mathrm{d}\tau \nonumber\\
            &\quad +\int_{\frac{t}{2}}^t\|DG_{22}^\ell(t-\tau)\ast D^2f_2(\tau)\|_{L^2}\mathrm{d}\tau \nonumber\\
			&\leq CN_0(1+t)^{-2}
			+C\int^{\frac{t}{2}}_0 (1+t-\tau)^{-2}(\|(u\cdot \nabla \phi)(\tau)\|_{L^2}+\|(\phi\text{div}u)(\tau)\|_{L^2})\mathrm{d}\tau
			\nonumber\\
			&\quad +C\int_{\frac{t}{2}}^t (1+t-\tau)^{-\frac{3}{2}}(\|D(u\cdot \nabla \phi)(\tau)\|_{L^2}+\|D(\phi\text{div}u)(\tau)\|_{L^2})\mathrm{d}\tau \nonumber\\
            &\quad+ C\int_0^{\frac{t}{2}} (1+t-\tau)^{-\frac{5}{2}}(\|(u\cdot \nabla u)(\tau)\|_{L^2}+\|(\phi \nabla \phi)(\tau)\|_{L^2})\mathrm{d}\tau\nonumber\\
         &\quad+C\int_{\frac{t}{2}}^t(1+t-\tau)^{-\frac{3}{2}}(\|D^2(u\cdot \nabla u)(\tau)\|_{L^2}+\|D^2(\phi \nabla \phi)(\tau)\|_{L^2})\mathrm{d}\tau \nonumber\\
			&\leq CN_0(1+t)^{-2}
			+\tilde{C}N_0^2 \int^{\frac{t}{2}}_0 (1+t-\tau)^{-2}(1+\tau)^{-\frac{7}{4}}\mathrm{d}\tau\nonumber\\
			&\quad +\tilde{C}N_0^2 \int_{\frac{t}{2}}^t (1+t-\tau)^{-\frac{3}{2}}(1+\tau)^{-\frac{9}{4}}\mathrm{d}\tau+ \tilde{C}N_0^2 \int^{\frac{t}{2}}_0 (1+t-\tau)^{-\frac{5}{2}}(1+\tau)^{-\frac{5}{4}}\mathrm{d}\tau \nonumber\\
            &\quad +\tilde{C}N_0^2 \int_{\frac{t}{2}}^t (1+t-\tau)^{-\frac{3}{2}}(1+\tau)^{-\frac{9}{4}}\mathrm{d}\tau \nonumber\\
			&\leq \tilde{C}(N_0+N_0^2)(1+t)^{-2}. \label{D3u-l}
		\end{align}     
		This together with \eqref{D3u-h}--\eqref{D3u-l}, yields
		\begin{equation*}
			\|D^3 u\|_{L^2} \leq \|D^3 u^h\|_{L^2}+\|D^3 u^{\ell}\|_{L^2}
			\leq \tilde{C} (N_0+N_0^2)(1+t)^{-2}.
		\end{equation*}
		Then the proof is completed.
	\end{proof}
	
	Combining the results in \rm{Lemmas} \ref{decay-lem-1}--\ref{decay-lem-2}, and noticing the equivalence of $\|D^k \phi\|_{L^2}$ and $\|D^k (\rho-1)\|_{L^2}$, we obtain the decay rates of $(\rho-1,u)$.
\begin{lem}\label{decay-lem-4}
		Under the {\it a priori} assumption \eqref{priori},
		\begin{align*}
			\|D^{k}(\rho-1)(t)\|_{L^2}&\leq \tilde{C}N_0(1+t)^{-\frac{k}{2}} ~\qquad\quad\text{for \,$k=0,1,2,3$,}\\
			\|D^{k} u(t)\|_{L^2}&\leq \tilde{C}N_0(1+t)^{-\frac{k+1}{2}} ~\qquad\,\,\text{for \,$k=0,1,2$,}\\
			\|D^3u(t)\|_{L^2}&\leq \tilde{C}(N_0+N_0^2)(1+t)^{-2}.
		\end{align*}
	\end{lem}
	We are now in a position to prove  Theorem \ref{main}.
	
\medskip
	\noindent\emph{\textbf{Proof of Theorem \ref{main}\rm{.}}} By the local existence (\rm{Theorem}~\ref{existence theorem}) and the regularity
	criterion (\rm{Theorem}~\ref{thm3.1}) for the Cauchy problem \eqref{euler}--\eqref{far},
	it suffices to close the {\it a priori} assumption \eqref{priori}.
	
	By the Sobolev interpolation and the equivalence \(\|\phi_0\|_{H^3}\sim\|\rho_0-1\|_{H^3}\)
	(valid under the uniform bounds on \(\rho\)), we have
	$$
	N_0=\|(\phi_0,u_0)\|_{H^3}\le \tilde{C}\|(\rho_0-1,u_0)\|_{H^3}\le \tilde{C}(M_0+\delta_0).
	$$
	Then, using \(H^2(\mathbb{R}^3)\hookrightarrow L^\infty\) and the Gagliardo--Nirenberg inequality,
	\begin{align*}
	\|\nabla \phi\|_{L^\infty}\ \lesssim\ \|D^2\phi\|_{L^2}^{\frac{1}{2}}\,\|D^3\phi\|_{L^2}^{\frac{1}{2}},
	\qquad
	\|\nabla u\|_{L^\infty}\ \lesssim\ \|D^2u\|_{L^2}^{\frac{1}{2}}\,\|D^3u\|_{L^2}^{\frac{1}{2}}.
	\end{align*}
	Hence, by \rm{Lemma}~\ref{decay-lem-4},
	\begin{align*}
		&\int_0^T\big(\|\nabla \phi\|_{L^\infty}+\|\nabla u\|_{L^\infty}\big)\,\mathrm{d}t\\
	   &\leq C \int_0^T \big(
		\| D^2 \phi\|_{L^2}^{\frac{1}{2}}
		\|D^3 \phi\|_{L^2}^{\frac{1}{2}}
		+\|D^2 u\|_{L^2}^{\frac{1}{2}}
		\|D^3 u\|_{L^2}^{\frac{1}{2}} \big)\mathrm{d}t\\
		&\leq C \int_0^T \big(
		\|D^2 \phi\|_{L^2}^{\frac{1}{2}}
		\|D^3 \phi\|_{L^2}^{\frac{3}{8}}
		\|D^3\phi\|_{L^2}^{\frac{1}{8}}
		+\|D^2 u\|_{L^2}^{\frac{1}{2}}
		\|D^3 u\|_{L^2}^{\frac{1}{3}}
		\|D^3 u\|_{L^2}^{\frac{1}{6}} \big)\mathrm{d}t\\
		&\leq \tilde{C} \int_0^T (1+t)^{-\frac{1}{2}}
		(1+t)^{-\frac{9}{16}}\delta_0^{\frac{1}{8}}
		N_0^{\frac{7}{8}} \mathrm{d}t
		+ \tilde{C} \int_0^T (1+t)^{-\frac{3}{4}}
		(1+t)^{-\frac{2}{3}}\delta_0^{\frac{1}{6}}N_0^{\frac{1}{2}}
		(N_0+N_0^2)^{\frac{1}{3}}\mathrm{d}t\\
		&\leq \tilde{C}\delta_0^{\frac{1}{8}}
		N_0^{\frac{7}{8}} \int_0^T (1+t)^{-\frac{17}{16}}\mathrm{d}t
		+\tilde{C}\delta_0^{\frac{1}{6}}
		N_0^{\frac{1}{2}}(N_0+N_0^2)^{\frac{1}{3}} \int_0^T (1+t)^{-\frac{17}{12}}\mathrm{d}t
		\\
		&\leq \tilde{C}\delta_0^{\frac{1}{8}}(M_0+\delta_0)^{\frac{7}{8}}+\tilde{C}\delta_0^{\frac{1}{6}}(M_0+\delta_0)^{\frac{1}{2}}(M_0+\delta_0+M_0^2+\delta_0^2)^{\frac{1}{3}}.
	\end{align*}
	According to the condition that $0<M_0\leq\delta_0^{-\frac{1}{11}}$, we obtain from the smallness of  $\delta_0$ that
\begin{align*}
	\int_0^T (\|\nabla \phi\|_{L^{\infty}}
	+\|\nabla u\|_{L^{\infty}})\mathrm{d}t
	\leq C\delta_0^{\frac{1}{22}}+C\delta_0^{\frac{2}{33}}\leq \frac{1}{2}\delta.
\end{align*}
This closes the {\it a priori} assumption \eqref{priori}, so the solution extends globally.

Therefore, for the  isentropic case $\gamma>1$, the Cauchy problem
	\eqref{euler}--\eqref{far} admits a unique global classical solution
	\((\rho-1,u)\in C([0,\infty);H^3(\mathbb{R}^3))\),
	and the decay rates in \rm{Lemma}~\ref{decay-lem-4} hold.
	\hfill\(\square\)

\section{Isothermal limit of solutions of the isentropic 
Euler equations}\label{s2}
In this section, we establish the isothermal limit of the solutions of 
the Cauchy problem of the isentropic compressible Euler equations 
with damping under the framework of classical solutions.

\medskip
\noindent{\it\textbf{Proof of Theorem \ref{thm:gamma_to_1}\rm{.}}} 
We add a subscript $\gamma$ to $(\phi,u)$ to indicate its dependence on $\gamma$ and rewrite \eqref{eq:isentropic} as
		\begin{equation}\label{eq:sym_gamma}
			\left\{
			\begin{aligned}
				&\partial_t \phi_\gamma + \bar\sigma(\gamma)\,\mathrm{div}u_\gamma
				= -\,u_\gamma\cdot\nabla\phi_\gamma - \frac{\gamma-1}{2}\,\phi_\gamma\,\mathrm{div}u_\gamma,\\[2pt]
				&\partial_t u_\gamma + \bar\sigma(\gamma)\,\nabla\phi_\gamma + u_\gamma
				= -\,u_\gamma\cdot\nabla u_\gamma - \frac{\gamma-1}{2}\,\phi_\gamma\,\nabla\phi_\gamma .
			\end{aligned}\right.
		\end{equation}
Note that all the estimates in \S~\ref{s1} are uniform with respect 
to $\gamma \in (1,3]$. In particular, using \eqref{aH3}, we derive 
$$
\sup\limits_{t\in[0,\infty)}\|(\phi_\gamma,u_\gamma)(t)\|_{H^3}\leq C,
$$
which, together with \eqref{eq:sym_gamma}, yields 
$$\sup\limits_{t\in[0,\infty)}\|(\partial_t\phi_\gamma,\partial_t u_\gamma)(t)\|_{H^2}\leq C.$$	
Thus, we obtain from Lemma \ref{AL} that, after further extracting a subsequence,
\begin{equation*}
  (\phi_\gamma, u_\gamma)
  \rightarrow (\ln \rho-\ln \rho_*, u)
  \qquad\,\,
  \mathrm{strongly}\;\mathrm{in}\;\, C([0,T];H^{3-\eta}_{\mathrm{loc}}(\mathbb{R}^3))
\end{equation*}
for any $\eta\in(0,3)$ and any finite time $T>0$. After the standard diagonal argument,
we further obtain
\begin{equation*}
	(\phi_\gamma, u_\gamma)
	\rightarrow (\ln \rho-\ln \rho_*, u)
	\qquad\,\, \mathrm{strongly}\;\mathrm{in}\;\, C([0,+\infty);H^{3-\eta}_{\mathrm{loc}}(\mathbb{R}^3)).
\end{equation*}
Moreover, taking $\gamma \rightarrow 1$  in \eqref{eq:sym_gamma}, 
we obtain the strong limit that is the classical solution of the Cauchy
problem of the isothermal compressible Euler equations with damping \eqref{au0}.
Finally, the uniqueness of solutions of the Cauchy problem 
of the limit system \eqref{au0} means that the above convergence 
holds without restricting to a subsequence.
Thus, we complete the proof of Theorem \ref{thm:gamma_to_1}. 	

\hfill\(\square\)
	
	\medskip
	\indent
	{\bf Acknowledgements:}
The research of Gui-Qiang G. Chen was supported in part by the UK Engineering and
Physical Sciences Research Council under Awards EP/L015811/1 and EP/V008854. The research of Feimin Huang was supported in part by the National
Key R\&D Program of China under Grant No. 2021YFA1000800. The research of
Houzhi Tang was supported in part by the National Natural Science Foundation of
China under Grant No.12501293 and the Anhui Provincial Natural Science Foundation
under Grant No.2408085QA031. The research of Shuxing Zhang was supported in part
by the National Natural Science Foundation of China under Grant No.12501300,
the Basic Research Program of Jiangsu Province under Grant No.BK20250881, and the Natural Science Foundation of the Jiangsu Higher
Education Institutions of China under Grant No.25KJB110001.  The
research of Weiyuan Zou was supported in part by the China Scholarship Council
under Grant No.202406880015 and the National Natural Science Foundation of China
under Grant No.12001033.
	\bibliographystyle{plain}
	
\end{document}